\documentclass[12pt]{amsart}
\usepackage{amsmath,amsfonts,amssymb,amsxtra,amscd,enumerate,amsthm}
\usepackage{color}
\usepackage{latexsym}
\usepackage{epsfig}
\usepackage{fullpage}

\renewcommand\l{\lambda}

\newcommand\les{\lesssim}
\newcommand\ges{\gtrsim}

\newcommand{\lan}{\langle}
\newcommand{\ran}{\rangle}

\newcommand\R{\mathbb{R}}
\newcommand\C{\mathbb{C}}
\newcommand\Z{\mathbb{Z}}
\newcommand\N{\mathbb{N}}
\renewcommand\S{\mathbb{S}}
\renewcommand\H{\mathbb{H}}

\newcommand{\calO}{\mathcal O}
\newcommand\W{\mathcal{W}}

\newcommand\la{\langle}
\newcommand\ra{\rangle}
\newtheorem{theo}{Theorem}
\numberwithin{theo}{section} 
\newtheorem{lema}[theo]{Lemma}
\newtheorem{prop}[theo]{Proposition}

\newtheorem*{rema}{Remark}

\numberwithin{equation}{section}

\newcommand{\dHe}{\dot{H}^1_{e}}

\begin{document}
\title{Equivariant Schr\"odinger Maps in two spatial dimensions}

\author{I. Bejenaru} \address{ Department of Mathematics, University
  of Chicago, 5734 S. University Ave, Chicago, IL 60637}
\email{bejenaru@math.uchicago.edu}

\author{A. Ionescu} \address{Department of Mathematics, Princeton
  University, Washington Rd., Princeton, NJ 08540}
\email{aionescu@math.princeton.edu}

\author{C. Kenig} \address{ Department of Mathematics, University of
  Chicago, 5734 S. University Ave, Chicago, IL 60637}
\email{cek@math.uchicago.edu}

\author{D.\ Tataru} \address{Department of Mathematics, The University
  of California at Berkeley, Evans Hall, Berkeley, CA 94720, U.S.A.}
\email{tataru@math.berkeley.edu}

\thanks{I.B. was supported in part by NSF grant DMS-1001676. A. I. was
  partially supported by a Packard Fellowship and NSF grant
  DMS-1065710. C.K. was supported in part by NSF grant DMS-0968742. 
  D.T. was supported in part by the Miller Foundation and by NSF 
grant DMS-0801261}

\begin{abstract} We consider equivariant solutions for the
  Schr\"odinger map problem from $\R^{2+1}$ to $\S^2$ with energy less
  than $4\pi$ and show that they are global in time and scatter.

\end{abstract}

\maketitle

\section{Introduction}

In this article we consider the Schr\"odinger map equation in
$\R^{2+1}$ with values into $\S^2 \subset \R^3$ (the two dimensional
unit sphere in $\R^3$),
\begin{equation}
  u_t = u \times \Delta u, \qquad u(0) = u_0
  \label{SM}\end{equation}

This equation admits a conserved energy,
\[
E(u) = \frac12 \int_{\R^2} |\nabla u|^2 dx
\]
and is invariant with respect to the dimensionless scaling
\[
u(t,x) \to u(\lambda^2 t, \lambda x).
\]
The energy is invariant with respect to the above scaling, therefore
the Schr\"odinger map equation in $\R^{2+1}$ is {\em energy critical}.

The local theory for classical data was established in \cite{SuSuBa}
and \cite{Ga}. We will use the following

\begin{theo}[McGahagan, \cite{Ga}] \label{clasic} If $u_0 \in \dot H^1 \cap \dot
  H^3$ then there exists a time $T>0$, such that \eqref{SM} has a
  unique solution in $L^\infty_t \dot ([0,T]:\dot H^1 \cap \dot H^3)$.
\end{theo}

The local and global in time of the Schr\"odinger map problem with
small data has been intensely studied, see \cite{Be}, \cite{Be2},
\cite{bik}, \cite{BIKT}, \cite{csu}, \cite{IoKe2}, \cite{IoKe3}.  The
state of the art result for the problem with small data was established by
the authors in \cite{BIKT} where they proved that classical solutions
(and in fact rough solutions too) with small energy are global in
time. More recently Smith established, in \cite{Sm1} and \cite{Sm2}, 
the global existence  in two dimensions of smooth Schr\"odinger maps with energy-dispersed data 
(essentially $\| u_0 \|_{\dot B^{1}_{2,\infty}} \ll 1$)
and satisfying $E(u) < 4\pi$, see the
commentaries below for the relevance of this threshold.

To understand the large data problem, one needs to describe the
solitons for \eqref{SM}. The solitons for this problem are the
harmonic maps, which are solutions to $u \times \Delta u=0$. The
harmonic maps cannot have arbitrary energy.  The trivial solitons are
of the form $u=Q$ for some $Q \in \S^2$ and their energy is $0$. The
next energy level admissible for solitons is $4 \pi$.

In this article we confine ourselves to the class of {\em equivariant}
Schr\"odinger maps.  These are indexed by an integer $m$ called the
equivariance class, and consist of maps of the form
\begin{equation} \label{equiv} u(r,\theta) = e^{m \theta R} \bar{u}(r)
\end{equation} 
Here $R$ is the generator of horizontal rotations, which can be
interpreted as a matrix or, equivalently, as the operator below
\[
R = \left( \begin{array}{ccc} 0 & -1 & 0 \\ 1 & 0 & 0 \\ 0 & 0 &
    0 \end{array} \right) , \qquad R u = \overrightarrow{k} \times u.
\]
Here and thereafter we denote by $\overrightarrow{i},
\overrightarrow{j}, \overrightarrow{k}$ the standard orthonormal basis
in $\R^3$, i.e. the vectors with coordinate representation $(1,0,0),
(0,1,0)$ respectively $(0,0,1)$. The case $m=0$ corresponds to radial
symmetry.

The energy for equivariant maps takes the following form:
\begin{equation} \label{energy} E(u) = \pi \int_{0}^\infty \left(
    |\partial_r \bar{u}(r)|^2 +
    \frac{m^2}{r^2}(\bar{u}_1^2(r)+\bar{u}_2^2(r)) \right) r dr
\end{equation}
If $m \ne 0$, then $E(u) < \infty$ implies that $u_1$ and $u_2$ have
limit zero as $r \rightarrow 0$ and $r \rightarrow \infty$.  For a
proof see for instance \cite{gkt1}. Due to the restriction on the size
of the energy, this implies that $u_3(t,r)$ has the same limit $+1$ or
$-1$ both as $r \rightarrow 0$ and $r \rightarrow \infty$, see
\cite{gkt1} or the commentaries in subsection \ref{COG}. To fix
matters we agree that this limit is $-1$ for all $t$ (the fact that
$u_3(0,t)$ or $u_3(\infty,t)$ cannot jump in time from $-1$ to $+1$ is
justified in subsection \ref{COG}).

The global regularity question in the case $m=0$, corresponding to
radial symmetry, has been considered recently by Gustafson and Koo,
see \cite{gk}. In this paper we consider the case $m=1$. More
precisely, our main result is the following:

\begin{theo} \label{MT} i) Let $u_0 \in \dot H^1
  \cap \dot H^3$ be a $1$-equivariant function satisfying $E(u_0) < 4
  \pi$.  Then \eqref{SM} has a unique global in time solution $u \in
  L^\infty(\R:\dot H^1 \cap \dot H^3)$.  In addition $\nabla u$, expressed in a
  Coulomb frame, scatters to the free solution of a suitable
  linear Schr\"odinger equation.
  
  ii) The above solution is Lipschitz continuous with respect to the initial
  data in $\dot H^1$. In particular if $u_0 \in \dot H^1$ 
  is a $1$-equivariant function satisfying $E(u_0) < 4
  \pi$ then  there is a global solution $u(t) \in L^\infty \dot H^1$ defined as the unique 
  limit of smooth solutions in $\dot H^1 \cap \dot H^3$.  
\end{theo}
The statement of the scattering cannot be made precise at this
time. We need to introduce a moving frame on $\S^2$, write the
equation of the coordinates of $\nabla u$ in that frame and identify
there the linear part of the Schr\"odinger equation. This will be
carried out in Section \ref{seccoulomb}.

The result in Theorem \ref{MT} is sharp from the following point of
view. Maps with energy less than $4\pi$ have topological degree $0$,
and cannot cover the full sphere. However, for energies 
greater than or equal to  $4\pi$ there are also maps with topological degree $1$
which fully cover the sphere. The solitons have least energy, i.e. $4\pi$,
among such maps, and all  degree one maps with energy slightly 
larger than $4\pi$ must stay close to the soliton family. 
The global in time behavior of such maps was first studied 
by  Bejenaru and Tataru~\cite{BT-SSM}, who proved that the solitons are unstable, but 
global solutions always exist for equivariant data which are small localized 
perturbations of solitons. Later Merle, Raphael and Rodnianski~\cite{MRR}
produced examples of equivariant data which are somewhat less localized  
soliton perturbations, which still have energy slightly larger than $4 \pi$  and for which the solutions blow up in finite time.
One should note that the situation changes for larger $m$, namely $m \geq 3$;
in that case  Gustafson and collaborators, see \cite{gkt1},
\cite{gkt} and \cite{gnt} have established the stability of the corresponding
solitons in the $m$-equivariant class.

However, if one restricts attention to maps to $1$-equivariant maps
whose topological degree is $0$, then such maps cannot cover the
sphere if their energy is below $8 \pi$. Hence it is natural to
conjecture that for degree zero maps the above result should be valid
up to an $8\pi$ energy. This is why most of our arguments are
written for the $8\pi$ threshold. While going above $4\pi$
is not very difficult, extending the proof up to 
$8\pi$ requires some new idea; our present argument fails
 due to the lack of sign for a term in the virial identities.

Our result in Theorem \ref{MT} extends to all other $m \ne
0$. Moreover we expect that it extends to the case when the target
manifold is the hyperbolic space $\H^2$ in which case no restriction
on the data is needed due to the absence of nontrivial solitons.

\subsection{Definitions and notations.}
\label{defnot}

While at fixed time our maps into the sphere are functions defined on
$\R^2$, the equivariance condition allows us to reduce our analysis to
functions of a single variable $|x|=r \in [0,\infty)$.  One such
instance is exhibited in \eqref{equiv} where to each equivariant map
$u$ we naturally associate its radial component $\bar u$.  Some other
functions will turn out to be radial by definition, see, for instance,
all the gauge elements in Section \ref{seccoulomb}.  We agree to
identify such radial functions with the corresponding one dimensional
functions of $r$.  Some of these functions are complex valued, and
this convention allows us to use the bar notation with the standard
meaning, i.e. the complex conjugate.

Even though we work mainly with functions of a single spatial variable
$r$, they originate in two dimensions. Therefore, it is natural to
make the convention that for the one dimensional functions all the
Lebesgue integrals and spaces are with respect to the $rdr$ measure,
unless otherwise specified.
 
Since equivariant functions are easily reduced to their
one-dimensional companions via \eqref{equiv}, we introduce the one
dimensional equivariant version of $\dot H^1$,
\begin{equation}\label{defhe}
  \| f \|_{\dHe}^2 = \| \partial_r f\|_{L^2(rdr)}^2 + \| r^{-1} f \|_{L^2(rdr)}^2.
\end{equation}
This is natural since for functions $u: \R^2 \to \R^2$ with
$u(r,\theta) = e^{\theta R} \bar u(r)$ (here $R u = \overrightarrow{k}
\times u$ or, as a matrix, it is the upper left $2 \times 2$ block of
the original matrix $R$) we have
\[
\| u \|_{\dot H^1}= (2\pi)^\frac12 \| \bar u \|_{\dHe}.
\]
It is important to note that functions in $\dHe$ enjoy the following
properties: they are continuous and have limit $0$ both at $r=0$ and
$r=\infty$, see \cite{gkt1} for a proof.

We introduce $\dot H_e^{-1}$ as the dual space to $\dHe$ with respect
to the $L^2$ pairing, i.e.
\[
\| f \|_{\dot H^{-1}_e} = \sup_{\| \phi \|_{\dHe}=1} \lan f,\phi \ran
\]
We will mostly be interested in elements from $\dot H^{-1}_e$ 
of the form $\partial_r g$ or $\frac{g}r$ with $g \in L^2$.

Three operators which are often used on radial functions are
$[\partial_r]^{-1}, [r^{-1} \bar \partial_r]^{-1}$ and
$[r \partial_r]^{-1}$ defined as
\[
\begin{split}
  [\partial_r]^{-1} f(r) & = - \int_{r}^\infty f(s) ds, \quad  [r^{-1}\bar \partial_r]^{-1} f(r) = \int_{0}^r f(s) sds \\
  [r \partial_r]^{-1}f(r) & = - \int_{r}^\infty \frac{1}s f(s) ds
\end{split}
\]
A direct argument shows that
\begin{equation} \label{rdrm}
  \begin{split}
    & \| [r\partial_r]^{-1}f \|_{L^p} \lesssim_p \| f \|_{L^p}, \qquad 1 \leq p < \infty, \\
    & \| r^{-2} [r^{-1} \bar \partial_r]^{-1}f \|_{L^p} \lesssim_p \| f \|_{L^p}, \qquad 1 < p \leq \infty, \\
    & \| [\partial_r]^{-1}f \|_{L^2} \lesssim \| f \|_{L^1}.
  \end{split}
\end{equation}

The equivariance properties of the functions involved in this paper
require that the two-dimensional Fourier calculus is replaced by the
Hankel calculus for one-dimensional functions which we recall below.

For $k \geq 0$ integer, let $J_k$ be the Bessel function of the first
kind,
\[
J_k(r)= \frac1\pi \int_0^\pi \cos(n\tau-r\sin\tau) d\tau
\]
If $H_k=\partial_r^2 + \frac1r \partial_r - \frac{k^2}{r^2}$, then
$J_k$ solves $H_k J_k = -J_k$.

We recall some formulas involving Bessel functions
\begin{equation} \label{derB}
  \partial_r J_k = \frac12 (J_{k-1}-J_{k+1}), \quad 
  (r^{-1}\partial_r)^m \left( \frac{J_k}{r^k} \right) = (-1)^m \frac{J_{k+m}}{r^{k+m}}, 
\end{equation}
where $J_{-k}=(-1)^k J_k$.

For each $k \geq 0$ integer one defines the Hankel transform $\mathcal
F_k$ by
\[
\mathcal F_k f (\xi) = \int_0^\infty J_k(r \xi) f(r) rdr
\]
The inversion formula holds true
\[
f(r) = \int_0^\infty J_k(r\xi) \mathcal F_k f(\xi) \xi d\xi
\]
The Plancherel formula holds true, hence in particular, the Hankel
transform is an isometry.

For a radial function $f$ and for an integer $k$ we define its
two-dimensional extension
\begin{equation} \label{Ddef} R_k(r,\theta)=e^{ik\theta} f(r)
\end{equation}
If $f \in L^2$ then $R_k f \in L^2$; if $R_k f$ has additional
regularity, this is easily red in terms of $\mathcal F_k f$.  Indeed
for any $s \geq 0$ integer the following holds true
\begin{equation} \label{Freg} R_k f \in \dot H^s \Leftrightarrow \xi^s
  \mathcal F_k f \in L^2
\end{equation}
For even values of $s$ this is a consequence of $\Delta R_k f= R_k H_k
f$, while for odd values of $s$ it follows by interpolation.

By direct computation, we also have that for $k \ne 0$,
\begin{equation} \label{Frg1} R_k f \in \dot H^1 \Leftrightarrow f \in
  \dHe, \qquad R_0 f \in \dot H^1 \Leftrightarrow \partial_r f \in
  L^2.
\end{equation}

We will use the following result
\begin{lema} \label{LBE}

  i) If $f \in L^2$ is such that $H_0 f \in L^2$, then the following
  holds true
  \[
  \|\partial_r^2 f \|_{L^2} + \| \frac{\partial_r f}{r} \|_{L^2} \les \| H_0
  f \|_{L^2}
  \]
  ii) If $f \in L^2$ is such that $H_2 f \in L^2$, then the following
  holds true
  \[
  \|\partial_r^2 f\|_{L^2}
 + \| \frac{\partial_r f}{r}\|_{L^2}+ \|\frac{f}{r^2}
  \|_{L^2} \les \| H_2 f \|_{L^2}
  \]
\end{lema}

\begin{proof} i) The proof follows the same lines as the one in ii),
  though it is easier.

  ii) Based on the representation formula
  \[
  f = \int J_2(r\xi) \mathcal F_2 f (\xi) \xi d\xi
  \]
  and by using \eqref{derB}, we compute
  \[
  \partial_r f = \frac12 \int (J_1(r\xi) ) - J_3(r\xi))\xi \mathcal
  F_2 f_0^- (\xi) \xi d\xi =g_1 - g_3
  \]
  with $R_1 g_1, R_3 g_3 \in \dot H^1$, hence $\partial_r f \in
  \dHe$. Using \eqref{Frg1} we obtain the conclusion for $\partial_r^2
  f$ and $\frac1r \partial_r f$. Since the estimate holds true for $H_2
  f$, it follows for $\frac{1}{r^2} f$ too.
\end{proof}

\section{The Coulomb gauge representation of the equation}
\label{seccoulomb}

In this section we rewrite the Schr\"odinger map equation for
equivariant solutions in a gauge form.  This approach originates in
the work of Chang, Shatah, Uhlenbeck~\cite{csu}. However, our analysis
is closer to the one in \cite{bik} and \cite{BT-SSM}.

\subsection{The Coulomb gauge} \label{CG}

The computations below are at the formal level as we are not yet
concerned with the regularity of the terms involved in writing various
identities and equations. Implicitly we use only the information $u
\in \dot H^1$. In subsection \ref{sreg} we prove that if $u \in \dot
H^3$ then all the gauge elements, their compatibility relations and
the equations they obey are meaningful in the sense that they involve
terms which are at least at the level of $L^2$.

We let the differentiation operators
$\partial_0,\partial_1,\partial_2$ stand for
$\partial_t, \partial_r, \partial_\theta$ respectively. Our strategy
will be to replace the equation for the Schr\"odinger map $u$ with
equations for its derivatives $\partial_1 u$, $\partial_2 u$ expressed
in an orthonormal frame $v,w \in T_u \S^2$. We choose $v \in T_u \S^2$
such that $v \cdot v=1$ and define $w = u \times v \in T_u \S^2$; to
summarize
\begin{equation}
  v \cdot v=1, \qquad v \cdot u =0, \qquad w=v \times w
\end{equation}
From this, we obtain
\begin{equation}
  w \cdot v=0, \qquad w \cdot w=1, \qquad v \times w= u, \qquad w \times u=v
\end{equation}
Since $u$ is $1$-equivariant it is natural to work with
$1$-equivariant frames, i.e.
\[
v = e^{\theta R} \bar{v}(r), \qquad w = e^{ \theta R} \bar{w}(r).
\]
where $\bar v, \bar w$ (as well as $\bar u$ from \eqref{equiv}) are
unit vectors in $\R^3$.

Given such a frame we introduce the differentiated fields $\psi_k$ and
the connection coefficients $A_k$ by
\begin{equation}\label{connection}
  \begin{split}
    \psi_k = \partial_k u \cdot v + i \partial_k u \cdot w, \qquad A_k
    = \partial_k v \cdot w.
  \end{split}
\end{equation}
Due to the equivariance of $(u,v,w)$ it follows that both $\psi_k$ and
$A_k$ are spherically symmetric (therefore subject to the conventions
made in Section \ref{defnot}). Conversely, given $\psi_k$ and $A_k$ we
can return to the frame $(u,v,w)$ via the ODE system:
\begin{equation}
  \label{return}
  \left\{ \begin{array}{l}
      \partial_k u = (\Re {\psi_k}) v  + (\Im{\psi_k}) w
      \cr
      \partial_k v = - (\Re{\psi_k}) u + A_k w
      \cr
      \partial_k w = - (\Im{\psi_k}) u - A_k v
    \end{array} \right.
\end{equation}

If we introduce the covariant differentiation
\[
D_k = \partial_k + i A_k, \ \ k \in \{0,1,2\}
\]
it is a straightforward computation to check the compatibility
conditions:
\begin{equation} \label{compat} D_l \psi_k = D_k \psi_l, \ \ \
  l,k=0,1,2.
\end{equation}
The curvature of this connection is given by
\begin{equation}\label{curb}
  D_l D_k - D_k D_l = i(\partial_l A_k - \partial_k
  A_l) = i \Im{(\psi_l \bar{\psi}_k)}, \ \ \  l,k=0,1,2.
\end{equation}
An important geometric feature is that $\psi_2, A_2$ are closely
related to the original map. Precisely, for $A_2$ we have:
\begin{equation}
  A_2 = ( \overrightarrow{k} \times v) \cdot w=  \overrightarrow{k} \cdot (v \times w)= \overrightarrow{k} \cdot u =  u_3
  \label{a2u3}\end{equation}
and, in a similar manner,
\begin{equation}
  \psi_2 = w_3-iv_3
  \label{psi2vw3}\end{equation}
Since the $(u,v,w)$ frame is orthonormal, it follows that $|\psi_2|^2
= u_1^2 + u_2^2$ and the following important conservation law
\begin{equation} \label{cons} |\psi_2|^2 + A_2^2 = 1
\end{equation}

Now we turn our attention to the choice of the $(\bar v,\bar w)$ frame
at $\theta = 0$. Here we have the freedom of an arbitrary rotation
depending on $t$ and $r$. In this article we will use the Coulomb
gauge, which for general maps $u$ has the form $\text{div } A = 0$.
In polar coordinates this is written as $\partial_1 A_1 + r^{-2}
\partial_2 A_2 = 0$. However, in the equivariant case $A_2$ is radial,
so we are left with a simpler formulation $A_1 = 0$, or equivalently
\begin{equation}
  \partial_r \bar v \cdot \bar w=0
  \label{coulomb}\end{equation}
which can be rearranged into a convenient ODE as follows
\begin{equation} \label{cgeq}
  \partial_r  \bar v = (\bar v \cdot \bar u) \partial_r \bar u
  - (\bar v \cdot \partial_r \bar u) \bar u
\end{equation}
The first term on the right vanishes and could be omitted, but it is
convenient to add it so that the above linear ODE is solved not only
by $\bar v$ and $\bar w$, but also by $\bar u$. Then we can write an
equation for the matrix $ \calO = (\bar v, \bar w,\bar u)$:
\begin{equation} \label{cgeq-m}
  \partial_r \calO = M \calO, \qquad M = \partial_r \bar u \wedge \bar u : = 
  \partial_r \bar u \otimes \bar u - \bar u \otimes \partial_r \bar u
\end{equation}
with an antisymmetric matrix $M$.

An advantage of using the Coulomb gauge is that it makes the
derivative terms in the nonlinearity disappear. Unfortunately, this
only happens in the equivariant case, which is why in \cite{BIKT} we
had to use a different gauge, namely the caloric gauge.

The ODE \eqref{cgeq} needs to be initialized at some point.  A change
in the initialization leads to a multiplication of all of the $\psi_k$
by a unit sized complex number. This is irrelevant at fixed time, but
as the time varies we need to be careful and choose this
initialization uniformly with respect to $t$, in order to avoid
introducing a constant time dependent potential into the equations via
$A_0$.  Since in our results we start with data which converges
asymptotically to $-\vec{k}$ as $r \to \infty$, and the solutions
continue to have this property, it is natural to fix the choice of
$\bar{v}$ and $\bar{w}$ at infinity,
\begin{equation}
  \label{bcvw}
  \lim_{r \to \infty} \bar v(r,t) = \vec{i} , \qquad \lim_{r \to \infty} \bar w(r,t) =
  -\vec{j}
\end{equation}

The existence of a unique solution $\bar v \in C((0,+\infty):\R^3)$ of
\eqref{cgeq} satisfying \eqref{bcvw} is standard, we skip the details.
Moreover the solution is continuous with respect to $u$
in the following sense
\begin{equation} \label{gcont}
\| \bar v- \bar{\tilde v} \|_{L^\infty} \les \| u - \tilde u \|_{\dot H^1}
\end{equation}

\subsection{ Schr\"odinger maps in the Coulomb gauge} \label{COG}

We are now prepared to write the evolution equations for the
differentiated fields $\psi_1$ and $\psi_2$ in \eqref{connection}
computed with respect to the Coulomb gauge.

Writing the Laplacian in polar coordinates, a direct computation using
the formulas \eqref{connection} shows that we can rewrite the
Schr\"odinger Map equation \eqref{SM} in the form
\begin{equation}\label{psizero}
  \psi_0 = i \left(D_1 \psi_1 + \frac{1}{r} \psi_1 + \frac{1}{r^2} D_2
    \psi_2\right)
\end{equation}
Applying the operators $D_1$ and $D_2$ to both sides of this equation
and using the relation \eqref{curb} for $l,k = 1,2$ we obtain
\begin{equation} \label{psiab}
  \begin{split}
    D_1 \psi_0 = & \ i\left( D_1(D_1+\frac1r) \psi_1 + \frac{1}{r^2}
      D_2 D_1 \psi_2\right) - \frac{1}{r^2} \Im{(\psi_1 \bar{\psi}_2)}
    \psi_2
    \\
    D_2 \psi_0 = & \ i\left( (D_1+\frac1r)D_2 \psi_1 + \frac{1}{r^2}
      D_2 D_2 \psi_2\right) - \Im{(\psi_2 \bar{\psi}_1)} \psi_1
  \end{split}
\end{equation}
Using now \eqref{compat} for $(k,l)=(0,1)$ respectively $(k,l)=(0,2)$
on the left and for $(k,l)=(1,2)$ on the right we can derive the
evolution equations for $\psi_m$, $ m=1,2$:
\begin{equation} \label{dtpsiab}
  \begin{split}
    D_0 \psi_1 = & \ i\left( D_1(D_1+\frac1r) + \frac{1}{r^2} D_2
      D_2\right) \psi_1 - \frac{1}{r^2} \Im{(\psi_1 \bar{\psi}_2)}
    \psi_2
    \\
    D_0 \psi_2 = & \ i\left( (D_1+\frac1r)D_1 + \frac{1}{r^2} D_2
      D_2\right) \psi_2 - \Im{(\psi_2 \bar{\psi}_1)} \psi_1
  \end{split}
\end{equation}
% or in expanded form
% \begin{equation} \label{dtpsiab}
%   \begin{split}
%     \partial_t \psi_1 + i A_0 \psi_1 = & \ i \Delta \psi_1 - 2
%     A_1 \partial_1
%     \psi_1  - \partial_1 A_1 \psi_1 - \frac{1}{r} A_1 \psi_1  - i A_1^2 \psi_1\\
%     & \ - i \frac{1}{r^2} A_2^2 \psi_1 - i \frac1{r^2} \psi_1
%     + \frac2{r^3} A_2 \psi_2 - \frac{1}{r^2} \Im{(\psi_1 \bar{\psi}_2)} \psi_2 \\
%     \partial_t \psi_2 + i A_0 \psi_2 = & \ i \Delta \psi_2 - 2
%     A_1 \partial_1
%     \psi_2  - \partial_1 A_1 \psi_2 - \frac{1}{r} A_1 \psi_2 \\
%     & \ - i A_1^2 \psi_2 - i \frac{1}{r^2} A_2^2 \psi_2 -
%     \Im{(\psi_2 \bar{\psi}_1)} \psi_1
%   \end{split}
% \end{equation}
In our set-up all functions are radial and we are using the the
Coulomb gauge $A_1 = 0$. Then these equations take the simpler form
\begin{equation*} \label{smg}
  \begin{split}
    \partial_t \psi_1 + i A_0 \psi_1 = & i \Delta \psi_1 - i
    \frac{1}{r^2} A_2^2 \psi_1 -i \frac1{r^2} \psi_1
    +  \frac2{r^3} A_2 \psi_2 - \frac{1}{r^2} \Im{(\psi_1 \bar{\psi}_2)} \psi_2 \\
    \partial_t \psi_2 + i A_0 \psi_2 = & i \Delta \psi_2 - i
    \frac{1}{r^2} A_2^2 \psi_2 - \Im{(\psi_2 \bar{\psi}_1)} \psi_1
  \end{split}
\end{equation*}
The two variables $\psi_1$ and $\psi_2$ are not independent.  Indeed,
the relations \eqref{compat} and \eqref{curb} for $(k,l)=(1,2)$ give
\begin{equation} \label{comp}
  \partial_r A_2= \Im{(\psi_1 \bar{\psi}_2)}, \qquad \partial_r \psi_2 = i A_2
  \psi_1
\end{equation}
which at the same time describe the relation between $\psi_1$ and
$\psi_2$ and determine $A_2$.

From the compatibility relations involving $A_0$, we obtain
\begin{equation} \label{a0}
  \partial_r A_0= - \frac1{2r^2} \partial_r (r^2 |\psi_1|^2 - |\psi_2|^2)
\end{equation}
from which we derive
\begin{equation}
  A_0 = - \frac12 \left( |\psi_1|^2 - \frac{1}{r^2}|\psi_2|^2\right)
  - [r\partial_r]^{-1} \left( |\psi_1|^2 - \frac{1}{r^2}|\psi_2|^2\right) 
  \label{a0bis}\end{equation}
This is where the initialization of the Coulomb gauge at infinity is
important. It guarantees that $A_0 \in L^p$, provided that $|\psi_1|^2
- r^{-2}|\psi_2|^2 \in L^p$ for $1 \leq p < \infty$. In particular,
without any additional regularity assumptions, we know that $A_0 \in
L^1$.  A direct computation using integration by parts gives that
\begin{equation} \label{A0c} \int A_0(r) rdr =0.
\end{equation}

The system satisfied by $\psi_1$ and $\frac{\psi_2}r$ (this being in
fact the correct variable instead of $\psi_2$) is given by:
\begin{equation*}
  \begin{split}
    (i \partial_t + \Delta - \frac{2}{r^2}) \psi_1 + \frac{4i}{r^2}
    \frac{\psi_2}r= & A_0 \psi_1 + \frac{A_2^2-1}{r^2} \psi_1
    + 2i \frac{A_2+1}{r^3} \psi_2 - i \Im{(\psi_1 \frac{\bar{\psi}_2}r}) \frac{\psi_2}r \\
    (i \partial_t + \Delta - \frac{2}{r^2}) \frac{\psi_2}{r} -
    \frac{4i}{r^2} \psi_1= & A_0 \frac{\psi_2}r + \frac{A_2^2-1}{r^2}
    \frac{\psi_2}r - 2i \frac{A_2+1}{r^2} \psi_1- i
    \Im{(\frac{\psi_2}r \bar{\psi}_1)} \psi_1
  \end{split}
\end{equation*}
The problem with this system is that its linear part is not decoupled.
This can be remedied by a change of variables. Indeed consider
\[
\psi^-=\psi_1 -i \frac{\psi_2}{r}, \qquad \psi^+=\psi_1 + i
\frac{\psi_2}{r}
\]
It turns out that $\psi^\pm$ satisfy a similar system (described
below) whose linear part is decoupled.  The relevance of the variables
$\psi^{\pm}$ comes also from the following reinterpretation.  If
$\W^{\pm}$ is defined as the vector
\[
\W^{\pm} = \partial_r u \pm \frac{1}{r} u \times R u \in T_u(\S^2)
\]
then $\psi^\pm$ is the representation of $\W^\pm$ with respect to the
frame $(v,w)$. The vector field $\W^+$ vanishes if and only if $u$ is
an equivariant soliton which starts at the south pole at $r =0$ and
goes to the north pole at $r = \infty$. For $\W^-$ the poles get
interchanged.  On the other hand, a direct computation leads to
\[
\begin{split}
  E(u) & = \pi \int_0^\infty \left( |\partial_r \bar{u}|^2 +
    \frac{1}{r^2} |\bar{u} \times
    R \bar{u}|^2 \right) rdr \\
  & = \pi \| \bar \W^\pm \|_{L^2}^2 \mp 2 \pi (\bar u_3(\infty)- \bar
  u_3(0))
\end{split}
\]
where we recall that $u(r,\theta)= e^{m \theta R} \bar{u}(r)$ and
$\bar u_3(\infty)=\lim_{r \rightarrow \infty} \bar u_3(r), \bar
u_3(0)=\lim_{r \rightarrow 0} \bar u_3(r)$ are well-defined since
$\bar u_1, \bar u_2 \in \dHe$ and if $f \in \dHe$ then $\lim_{r
  \rightarrow 0} f(r)=\lim_{r \rightarrow 0} f(r)=0$, see \cite{gkt1}
or \cite{BT-SSM}. From the above computation and that $E(u) < 4\pi$ it
follows that $\bar u_3(0)=\bar u_3(\infty) \in \{-1,+1\}$ and we have
made the choice $\bar u_3(0)=\bar u_3(\infty)=-1$.  This limit is
invariant under the evolution. Indeed since $\bar u,\Delta \bar u \in
\dot H^1$ it follows that $\bar u \times \Delta \bar u \in \dHe$,
where we use the fact that $\dHe \subset L^\infty$ is an
algebra. Therefore $\bar u_t \in \dHe$ hence $\bar u(0,t), \bar
u(\infty,t)$ are continuous in time.
  
In our setup we have the following identity
\begin{equation} \label{basicpsi} \|\psi^\pm\|_{L^2}^2 = \| \bar
  \W^\pm \|_{L^2}^2 = \frac{E(u)}{\pi}.
\end{equation}

From \eqref{gcont} it follows that the following continuity property holds true
\begin{equation} \label{dpsi}
\| \psi^\pm - \tilde{\psi}^\pm \|_{L^2} \les \| u - \tilde u  \|_{\dot H^1}
\end{equation}

A direct computation yields the following system for $\psi^{\pm}$:
\[
\begin{split}
  ( i \partial_t + H^-) \psi^- & = \left( A_0 - 2 \frac{A_2+1}{r^2}
    + \frac{A_2^2-1}{r^2} - \frac1{r} \Im{(\psi_2 \bar{\psi}_1)} \right) \psi^- \\
  (i \partial_t + H^+) \psi^+ & = \left( A_0 + 2 \frac{A_2+1}{r^2} +
    \frac{A_2^2-1}{r^2} + \frac1{r} \Im{(\psi_2 \bar{\psi}_1)} \right)
  \psi^+
\end{split}
\]
where
\[
H^- = \Delta - \frac{4}{r^2}, \qquad H^+ = \Delta .
\]
Here and whenever $\Delta$ acts on radial functions, it is known that
$\Delta = \partial_r^2 + \frac1r \partial_r$.  By replacing $\psi_1 =
\psi^- + i r^{-1} \psi_2$ and using $A_2^2 + |\psi_2|^2=1$, we obtain
the key evolution system we work with in this paper,
\begin{equation} \label{psieq} \left\{ \begin{array}{l} (i \partial_t
      + H^{-}) \psi^- = (A_0 - 2 \frac{A_2+1}{r^2} -
      \frac1{r}\Im{(\psi_2 \bar{\psi}^-)}) \psi^- \cr (i \partial_t +
      H^+) \psi^+ = ( A_0 + 2 \frac{A_2+1}{r^2} + \frac1{r}\Im{(\psi_2
        \bar{\psi}^+)}) \psi^+
    \end{array} \right.
\end{equation}
We will use this system in order to obtain estimates for
$\psi^\pm$. The old variables $\psi_1$ and $\frac{\psi_2}r$ are
recovered from
\begin{equation} \label{rela} \psi_1 = \frac{\psi^+ + \psi^-}2, \qquad
  \frac{\psi_2}r = \frac{\psi^+ - \psi^-}{2i}
\end{equation}
From the compatibility conditions \eqref{comp} we derive the formula
for $A_2$
\begin{equation} \label{A2} A_2(r)+1 = - \int_0^r
  \frac{|\psi^+|^2-|\psi^-|^2}{4} s ds
\end{equation}

From \eqref{a0bis} $A_0$ is given by
\begin{equation} \label{A0} A_0 = - \frac12 \Re(\overline{\psi}^+
  \psi^-) + [r\partial_r]^{-1} \Re(\overline{\psi}^+ \psi^-)
\end{equation}

The compatibility condition \eqref{comp} reduces then to
\begin{equation} \label{compnew}
  \partial_r r (\psi^+ - \psi^-) = - A_2 (\psi^+ + \psi^-)
\end{equation}

Next assume that $\psi^\pm \in L^2$ are given such that $\| \psi^-
\|_{L^2}, \| \psi^+ \|_{L^2} < 2\sqrt{2}$ and satisfy the
compatibility conditions \eqref{compnew}.  We reconstruct $A_2,
\psi_2, \psi_1$ using the \eqref{rela} and \eqref{A2}.  From
\eqref{A2} and \eqref{compnew} it follows that \eqref{comp} hold true.
From \eqref{A2} it follows that $A_2 \in L^\infty$ and it is
continuous and has limits both at $0$ and $\infty$. From the
definition of $\psi_2$ we have $\frac{\psi_2}r \in L^2$ and from
\eqref{compnew} we derive $\partial_r \psi_2 \in L^2$, hence $\psi_2
\in \dHe$. From this and \eqref{A2} it follows that $\partial_r A_2
\in L^2$, while by invoking \eqref{rdrm} we obtain $\frac{A_2+1}r \in
L^2$, therefore $A_2+1 \in \dHe$.  In particular $A_2(\infty)=\lim_{r
  \rightarrow \infty} A_2(r)=-1$ which implies that $\| \psi^+
\|_{L^2}=\| \psi^- \|_{L^2}$.

If $\| \psi^+\|_{L^2}= \| \psi^- \|_{L^2}< 2$ it follows from
\eqref{A2} that
\begin{equation} \label{A2neg} \sup_{r \in (0,\infty)} A_2(r) \leq -1
  + \frac{\| \psi^{-} \|^2_{L^2}}4 < 0
\end{equation}

If $\| \psi^+\|_{L^2}= \| \psi^- \|_{L^2}< 2\sqrt{2}$ it follows from
\eqref{A2} that
\begin{equation} \label{A2les1} \sup_{r \in (0,\infty)} A_2(r) \leq -1
  + \frac{\| \psi^{-} \|^2_{L^2}}4 < 1
\end{equation}

Most of the arguments in this paper will work for $\| \psi^+\|_{L^2}=
\| \psi^- \|_{L^2}< 2\sqrt{2}$, but when involving the virial
identities we will need the stronger conclusion in \eqref{A2neg},
hence the limitation of our final result to the case $\|
\psi^+\|_{L^2}= \| \psi^- \|_{L^2}< 2$.

In fact one can keep track of a single variable, $\psi^-$ or $\psi^+$
since it contains all the information about the map, provided that the
choice of gauge \eqref{bcvw} was made. To be more precise,
\eqref{comp} gives the following
\begin{equation}
  \label{comp1}
  \partial_r A_2= \Im{(\psi^- \bar{\psi}_2)}+\frac{1}r |\psi_2|^2, \qquad \partial_r
  \psi_2 = i A_2 \psi^- - \frac{1}r A_2 \psi_2
\end{equation}
We will show that if $\| \psi^- \|_{L^2} < 2 \sqrt{2}$, this system has a
unique solution $A_2+1,\psi_2 \in \dHe$.  From this we can reconstruct
$\psi_1, \psi^+,A_0$. Then we can return to the map $u$ via the system \eqref{return}
with the boundary condition at infinity given by \eqref{bcvw}. A similar procedure can completely reconstruct $u$
from $\psi^+$.

Keeping track of both variables $\psi^\pm$ (instead of just one) via the system \eqref{psieq}
is advantageous for the dynamical (in time) properties of the problem. 
Understanding how to recover all the information from only
one variable, say $\psi^-$,  is advantageous for the elliptic part of the
profile decomposition in Proposition \ref{pcc}.

\subsection{Regularity of the gauge elements} \label{sreg} In this
section we clarify the regularity of the gauge elements.  We extend
$\psi^\pm$ to two-dimensional functions by
\begin{equation} \label{ext} R_+ \psi^+:=R_0 \psi^+, \qquad R_-
  \psi^-= R_2 \psi^-
\end{equation}
Essentially we claim that if $u \in \dot H^1 \cap \dot H^3$ then the
two-dimensional functions $R_\pm \psi^\pm \in H^2$.  Our main claim is
the following
\begin{prop} \label{greg} If $u \in \dot H^3$ then $R_\pm \psi^\pm \in
  H^2$ and
  \begin{equation} \label{rtr} \| u \|_{\dot H^1 \cap \dot H^3}
    \approx \| R_+ \psi^+ \|_{H^2} + \| R_- \psi^- \|_{H^2}
  \end{equation}
\end{prop}
The proof of this result will be provided in the Appendix.

Therefore, in the context of $u \in \dot H^1 \cap \dot H^3$, we have
that $R_\pm \psi^\pm \in H^2 \subset L^\infty$. The $H^2$ regularity
cannot be extended to (two-dimensional extensions of) $\psi_1$ and
$\frac{\psi_2}r$ since the $\psi^+$ and $\psi^-$ require different
phases for regularity. However, all the Sobolev embeddings are
inherited by $\psi_1$ and $\frac{\psi_2}r$, in particular $\psi_1,
\frac{\psi_2}r \subset L^\infty$.  Since $A_2=u_3$ it follows that
$A_2 \in \dot H^1 \cap \dot H^3$ and $\partial_t A_2 \in H^1$.
Finally by differentiating with respect to $t$ the system
\eqref{cgeq}, one can show that $\partial_t \bar v \in H^1$, hence
$A_0 \in H^1$ which in turn gives $\partial_r A_0 \in L^2$.  Therefore 
all the compatibility conditions in the previous two
subsections are at least at the level of $L^2$.

\subsection{Recovering the map from $\psi^-$.} In this section we
address the issue of re-constructing the Schr\"odinger map $u$
together with its gauge elements from only one of its reduced
variables, say $\psi^-$.  Reconstructing $\psi_2,A_2$ such that
$\psi_2,A_2+1 \in \dHe$ is a unique process; however, the
reconstruction of the actual map with its frame, i.e. of $(u,v,w)$ is
unique provided one prescribes conditions at $\infty$.  The map $u$
satisfies $u(\infty)=-\vec{k}$, while the gauge is subjected to the
choice \eqref{bcvw}.

The main result of this section is the following

\begin{prop} \label{recprop} Given $\psi^- \in L^2$, such that $\|
  \psi^- \|_{L^2} < 2\sqrt{2}$, there is a unique map $u: \R^2
  \rightarrow \S^2$ with the property that $\psi^-$ is the
  representation of $\mathcal{W}^-$ relative to a Coulomb gauge
  satisfying \eqref{bcvw}.  This also satisfies $E(u)=\pi \| \psi^-
  \|_{L^2}^2$.
  
  The map $\psi^- \rightarrow u$ is Lipschitz continuous in the following sense
\begin{equation} \label{udif}
\| u - \tilde u \|_{\dot H^1} \les \| \psi^- - \tilde \psi^{-} \|_{L^2}
\end{equation}

\end{prop}

Here $\psi^+$ can be reconstructed from $\psi^-$.  Moreover the
equations \eqref{comp1} which we use for reconstruction force the
compatibility condition \eqref{compnew} between $\psi^\pm$.  The
result remains true if we start from $\psi^+$ just that we would start
the reconstruction (described below) from the corresponding
\eqref{comp1} written for $\psi^+$. The two problems are in effect
equivalent via an inversion. The uniqueness of the reconstruction
guarantees that starting from either $\psi^+$ or $\psi^-$ (which are
assumed to be compatible) gives the same $u$.
 
The proof consists of several steps. The first one deals with
recovering the two gauge elements $\psi_2,A_2$ from $\psi^-$ by using
the system \eqref{comp1}.

\begin{lema} \label{a2p2c} Given $\psi^- \in L^2$, such that 
$\|  \psi^- \|_{L^2} < 2 \sqrt{2}$, the system \eqref{comp1} has a unique
  solution $(A_2, \psi_2)$ satisfying $\psi_2, A_2 +1 \in \dHe$. This
  solution satisfies
  \begin{equation} \label{2est} \| \psi_2 \|_{\dHe} + \| A_2+1
    \|_{\dHe} + \| \frac{A_2+1}r \|_{L^1(dr)} \les \| \psi^- \|_{L^2}
  \end{equation}
  In addition we have the following properties:
  
  i) given $\epsilon > 0$, and $R$  such that $\| \psi^-
  \|_{L^2(\R \setminus [R^{-1},R])} \leq \epsilon$, then the following holds
  true
  \begin{equation} \label{loc2est} \| \psi_2 \|_{\dHe(\R \setminus
      [\epsilon R^{-1},R])} + \| A_2 +1 \|_{\dHe(\R \setminus
      [\epsilon R^{-1},R])} \les \epsilon
  \end{equation}
  ii) if $(\tilde A_2, \tilde \psi_2)$ is another solution (as above)
  to \eqref{comp1} corresponding to $\tilde \psi^-$, then
  \begin{equation} \label{1estdif} \| \psi_2 - \tilde \psi_2 \|_{\dHe}
    + \| A_2 - \tilde A_2 \|_{\dHe} \les \| \psi^- - \tilde \psi^-
    \|_{L^2}
  \end{equation}
  iii) if $(\tilde A_2, \tilde \psi_2)$ satisfy $\tilde \psi_2, \tilde
  A_2+1 \in \dHe $ and solve
  \begin{equation} \label{sys2ap}
    \begin{split}
      \partial_r \tilde \psi_2 = & \ i \tilde A_2 \tilde \psi^- -
      \frac{1}r \tilde A_2 \tilde \psi_2 + E_1 \\
      \partial_r \tilde A_2= & \ \Im{(\tilde \psi^- \bar{\tilde
          \psi}_2)}+\frac{1}r (1- \tilde A_2^2) + E_2
    \end{split}
  \end{equation}
  where $\| |E_1|+ |E_2|\|_{L^1(dr)+L^2} \les \epsilon$ then
  \begin{equation} \label{2estdif} \| \psi_2 - \tilde \psi_2 \|_{\dHe}
    + \| A_2 - \tilde A_2 \|_{\dHe} \les \| \psi^- - \tilde \psi^-
    \|_{L^2}+ \epsilon
  \end{equation}
  
  iv) if $\psi^- \in L^p$ with $1 \leq p < \infty$ then $\psi^+,
  \frac{\psi_2}r,\frac{1+A_2}r \in L^p$ and
  \begin{equation} \label{strfix} \| \psi^+ \|_{L^p} + \|
    \frac{\psi_2}r \|_{L^p} + \| \frac{1+A_2}r \|_{L^p} \les \| \psi^-
    \|_{L^p}
  \end{equation}
  
  v) if $R_{-} \psi^- \in H^s$ then $R_+ \psi^+ \in H^s$ for any $s
  \in \{ 1,2,3 \}$.
\end{lema}

The reason for having the second type of statement in \eqref{2estdif}
is of technical nature and will be apparent in Section
\ref{concomp}. The equation for $\tilde A_2$ in \eqref{sys2ap} is more
convenient in that form when taking differences. For the original
system \eqref{comp1} it does not matter how one writes the equation
for $A_2$ thanks to the conservation law $|\psi_2|^2+ A_2^2=1$;
however in the case of \eqref{sys2ap} this conservation law does not
hold true, hence we write the system in the more convenient form
\eqref{sys2ap}.

\begin{proof}
  Our strategy is to solve the ode system \eqref{comp1} from infinity.
  Since $\psi_2, 1+ A_2 \in \dHe$, it follows that $\lim_{r
    \rightarrow \infty} \psi_2=0, \lim_{r \rightarrow \infty}
  A_2=-1$ These two conditions play the role of boundary conditions
  at infinity. Since $\partial_r (|\psi_2|^2+ A_2^2) = 0$, it follows
  from the conditions at $\infty$ that $|\psi_2|^2+ A_2^2=1$ holds on
  all of $\R_+$.  We define
  \[
  \psi_1 = \psi^-+ i \frac{\psi_2}r, \qquad \psi^+ = \psi^-+ 2i
  \frac{\psi_2}r
  \]
  and note that
  \[
  \partial_r A_2 = \frac14 r(|\psi^+|^2 - |\psi^-|^2)
  \]
  The function $A_2$ we seek must satisfy $1+A_2 \in \dHe$, therefore
  it needs to have limit zero at both $r=0$ and $r = \infty$.  Hence
  integrating from infinity it gives
  \[
  1+ A_2 \leq \frac14 \|\psi^{-}\|_{L^2}^2 < 2
  \]
  Thus we must have $\sup_{r \in (0,\infty)} A_2(r) < 1$ everywhere
  (which is to say that the corresponding map $u$ cannot reach the north pole).

  To establish existence and uniqueness for \eqref{comp1} it suffices
  to do so in a neighbourhood $[R,+\infty)$ of infinity. The extension
  of this down to $r=0$ follows from standard arguments since
  $L^2(rdr) \subset L^{1}_{loc}(dr)$.

  To prove existence, by choosing $R$ large enough we can assume
  without any restriction in generality that
  \begin{equation}\label{psismall}
    \| \psi^-\|_{L^2(R,\infty)}  \leq \epsilon
  \end{equation}
  and seek $(\psi_2, A_2)$ with the property that
  \begin{equation}\label{psi2small}
    \|\psi_2\|_{\dot H^1_e(R,\infty)} \lesssim \epsilon
  \end{equation}
  This implies that $|\psi_2| \lesssim \epsilon$ and, by the relation
  $|\psi_2|^2+ A_2^2 =1$, it also gives $A_2 + 1 \les
  \epsilon^2$. Hence we are allowed to substitute
  $A_2(r)=-\sqrt{1-|\psi_2(r)|^2}$ in the $\psi_2$ equation and
  discard the dependent $A_2$ equation. We rewrite the $\psi_2$
  equation as
  \[
  (\partial_r - \frac{1}{r}) \psi_2 = -i \psi^- + i \frac{A_2+1}r
  \psi^- - \frac{(A_2+1)\psi_2}{r^2}
  \]
  or equivalently
  \[
  r \partial_r \frac{\psi_2}{r} = -i \psi^-+ i(A_2+1) \psi^- -
  \frac{(A_2+1)\psi_2}{r}
  \]
  and further
  \[
  \psi_2 = -i r [r \partial_r]^{-1} \psi^-+ r [r \partial_r]^{-1}(i
  (A_2+1) \psi^- - \frac{(A_2+1)\psi_2}{r})
  \]
  We know from \eqref{rdrm} that $[r \partial_r]^{-1}$ maps $L^2$ to
  $L^2$, which easily implies that
  \[
  r [r \partial_r]^{-1}: L^2 \to \dot H^1_e
  \]
  Hence in order to obtain $\psi_2$ via the contraction principle it
  suffices to show that for $\psi$ as in \eqref{psismall} and $\psi_2$
  as in \eqref{psi2small} the map
  \[
  \psi_2 \to i (A_2+1) \psi^- - \frac{(A_2+1)\psi_2}{r}
  \]
  is Lipschitz from $\dot H^1_e \to L^2$ with a small ($O(\epsilon)$
  in this case) Lipschitz constant. But this is straightforward due to
  the embedding $\dot H^1_e \subset L^\infty$.  Thus the existence of
  $\psi_2$ in $[R,\infty)$ follows, and the corresponding $A_2$ is
  recovered via $A_2(r)=-\sqrt{1-|\psi_2(r)|^2}$.  The same argument
  also gives Lipschitz dependence of $\psi_2$ on $\psi^-$ in
  $[R,\infty)$.

  Once we have the global solution $(\psi_2,A_2)$ we obtain the bound
  \eqref{2est} via an energy type estimate. Precisely, denoting
  \[
  F = \frac{\psi_2}{1-A_2}
  \]
  its derivative satisfies
  \[
  \left|\frac{d}{dr} |F|^2 - \frac2r |F|^2\right| \lesssim
  \frac{|\psi^-|}{1-A_2} | F|
  \]
  Since $1-A_2$ is bounded from below, this further leads to
  \[
  \left|\frac{d}{dr} \frac{|F|}r\right| \lesssim \frac{|\psi^-|}{r}
  \]
  Integrating from infinity we obtain
  \[
  |F| \lesssim r [r \partial_r]^{-1} |\psi^-|
  \]
  Returning to $\psi_2$ we get the pointwise bound
  \begin{equation}\label{psi2point}
    |\psi_2| \lesssim  r [r \partial_r]^{-1} |\psi^-|
  \end{equation}
  By the $L^2$ boundedness of $ [r \partial_r]^{-1}$ this gives the
  $L^2$ bound on $\frac{\psi_2}{r}$. Then the $L^2$ bounds for
  $\partial_r \psi_2$ and $\partial_r A_2$ follow directly from
  \eqref{comp1}, while the $L^2$ and the $L^1(dr)$ bound for
  $\frac{1+A_2}r$ are consequences of the compatibility relation
  $|\psi_2|^2 + A_2^2=1$.

  The above argument already gives the $[R,\infty)$ part of
  \eqref{loc2est} since \eqref{psi2point} holds on any such interval.
  Getting the $(0,\epsilon R^{-1}]$ part of \eqref{loc2est} is
  slightly more delicate.  It suffices to get the $L^2$ bound for
  $\frac{\psi_2}{r}$.  From \eqref{psi2point} we have
  \[
  |\psi_2| \lesssim r [r \partial_r]^{-1} ( 1_{(0,R^{-1}]}|\psi^-|) +
  r [r \partial_r]^{-1} (1_{[R^{-1},\infty)}|\psi^-|)
  \]
  For the first term we use the smallness of $\psi_2$ in the
  hypothesis.  For the second we instead produce a pointwise bound
  using Cauchy-Schwarz:
  \[
  r [r \partial_r]^{-1} (1_{[R^{-1},\infty)}|\psi^-|) \lesssim r
  \int_{R^{-1}}^\infty s^{-1} |\psi^-(s)| ds \lesssim r R
  \|\psi_2\|_{L^2}, \qquad r < R^{-1}
  \]
  This implies the desired $L^2$ bound.

  The reason for this difference between the argument near infinity
  and the one near zero has to do with the following observation. If
  $\psi^-=0$, then the system \eqref{comp1} has many solutions which
  are completely described by $(\psi_2(r),A_2(r))=(e^{i\theta}
  h_1(\lambda r), h_3(\lambda r))$, for some $\lambda \in [0, \infty]$
  where
  \[
  h_1(r)= \frac{2r}{1+r^{2}}, \qquad h_3=\frac{r^{2}-1}{r^{2}+1}.
  \]
  These solutions are generated by the equivariant harmonic maps 
  (which solve $u \times \Delta u =0$)
  written in their Coulomb basis as described above. The main equivariant
  harmonic map is $Q(r,\theta)=e^{\theta R}(h_1(r),0,h_3(r))^T$ and
  all the other ones are of the form $Q(\lambda r, \theta + \alpha)$.
  Constructing the associated Coulomb gauge, as described above, we
  obtain $(\psi_2(r),A_2(r))=(e^{i\theta} h_1(\lambda r), h_3(\lambda
  r))$ See \cite{BT-SSM} for more details. These are not good global
  candidates for our system, unless $\lambda =0$ since $\lim_{r
    \rightarrow \infty} h_3(r)=1$. But they can be good local
  candidates near $r = 0$.

  Next we turn our attention to \eqref{1estdif} and \eqref{2estdif}.
  In fact, in the case of \eqref{1estdif}, in light of the
  conservation law $|\tilde \psi_2|^2 + \tilde A_2^2=1$,
  \eqref{1estdif} follows from \eqref{2estdif} with $E_1=E_2=0$. Hence
  we focus our attention on \eqref{2estdif}. We denote
  \[
  \delta \psi = \tilde{\psi} - \psi, \qquad \delta A_2=\tilde A_2-A_2
  , \qquad \delta \psi_2=\tilde \psi_2 - \psi_2
  \]
  Without any restriction in generality we can make the assumption $\|
  \delta \psi \|_{L^2} \ll 1$ and the bootstrap assumption
  \begin{equation} \label{boot2} \| \delta \psi_2\|_{L^\infty} + \|
    \delta A_2\|_{L^\infty} + \| \frac{\delta \psi_2}r\|_{L^\infty} +
    \|\frac{ \delta A_2}r\|_{L^\infty} \lesssim \epsilon^\frac12 + \|
    \delta \psi \|_{L^2}^\frac12
  \end{equation}

  Then we derive the equations for them modulo error terms.  We have
  \[
  \begin{split}
    \partial_r \delta \psi_2= & \ i \delta A_2 \tilde \psi^- + i A_2
    \delta \psi - \frac1r A_2 \delta \psi_2 - \frac1r \delta A_2
    \tilde \psi_2 + E_1
    \\
    \partial_r \delta A_2 = & \ \Im ( \psi^- \overline{\delta \psi_2})
    + \Im(\delta \psi \overline{\tilde \psi_2}) - \frac{2}{r} A_2
    \delta A_2 - \frac{1}{r} (\delta A_2)^2+ E_2
  \end{split}
  \]
  The following terms $i A_2 \delta \psi, \Im(\delta \psi
  \overline{\tilde \psi_2}) $ can be directly included into the error
  terms $E_1, E_2$, while the quadratic term $\frac{1}{r} (\delta
  A_2)^2$ can be included in the error term $E_2$ based on
  \eqref{boot2}.  We obtain the following linear system for $(\delta
  \psi_2,\delta A_2)$:
  \[
  \begin{split}
    \partial_r \delta \psi_2= & \ \frac1r \delta \psi_2+ i \tilde
    \psi^- \delta A_2 - \frac1r (A_2+1) \delta \psi_2 - \frac1r \delta
    A_2 \tilde \psi_2 + E_1
    \\
    \partial_r \delta A_2 = & \ \frac{2}{r} \delta A_2 + \Im ( \psi^-
    \overline{ \delta \psi_2}) - \frac{2}{r} (1+A_2) \delta A_2 + E_2
  \end{split}
  \]
  By considering the $\Re \delta \psi_2, \Im \delta \psi_2$
  separately, this is a system of the form
  \[
  \partial_r X = \frac1r LX + B X + F, \qquad
  L =\left( \begin{array}{ccc} 1 & 0 & 0 \\
      0 & 1 & 0 \\ 0 & 0 & 2 \end{array} \right)
  \]
  where the matrix $B,F$ satisfy $B \in L^2$ and $F \in L^2+ L^1(dr)$.
  This system needs to be solved with zero Cauchy data at infinity.
  For this system we need to establish the bound
  \begin{equation}
    \| X\|_{L^\infty} +  \| \frac{X}{r}\|_{L^2} \lesssim \| F\|_{L^2 + L^1(dr)}
  \end{equation}
  If $B=0$ then
  \[
  X =\left( \begin{array}{ccc} r [r \partial_r]^{-1}  & 0 & 0 \\
      0 & r [r \partial_r]^{-1} & 0 \\ 0 & 0 & r^2
      [r^2 \partial_r]^{-1} \end{array} \right) F
  \]
  and the conclusion easily follows from argument of type
  \eqref{rdrm}.  If $B$ is small in either $L^2(rdr)$ or in $r^{-1}
  L^\infty$ then we can treat the $BX$ term perturbatively. If $B$ is
  large then some more work is needed.  We decompose $B = B_1+B_2$
  where $B_1 \in L^1(dr)$ and $|B_2| \ll \frac{1}r $.  We can
  construct the bounded matrix $e^{\int B_1}$ as a solution of
  $\partial_r e^{\int B_1}= e^{\int B_1} B_1$ which also has a bounded
  inverse. Then we can eliminate $B_1$ by conjugating with respect to
  $e^{\int B_1}$, and then treat the part with $B_2$ perturbatively.

  iv) From \eqref{psi2point} and \eqref{rdrm} we obtain
  \[
  \| \frac{\psi_2}{r} \|_{L^p} \les \| \psi^- \|_{L^p}
  \]
  from which \eqref{strfix} follows since $\psi^+= 2i \frac{\psi_2}r +
  \psi^-$ and $1+A_2=\frac{|\psi_2|^2}{1-A_2}$.

  v) Throughout this argument, the use of Sobolev embedding refers to
  the two-dimensional standard Sobolev embeddings which apply to
  $R_\pm \psi^\pm$, which then can be read in terms of $\psi^\pm$.

  If $s=1$ then we have
  \[
  \partial_r \psi^+ = 2i \partial_r \frac{\psi_2}r + \partial_r \psi^-
  = \frac{2i}r (i A_2 \psi^- - \frac1r (A_2+1)\psi_2 )+ \partial_r
  \psi^-
  \]
  Using \eqref{strfix} we estimate
  \[
  \begin{split}
    \| \partial_r \psi^+ \|_{L^2} & \les \| A_2 \|_{L^\infty} \|
    \frac{\psi^-}r \|_{L^2} + \|\frac{A_2+1}r\|_{L^4} \|
    \frac{\psi_2}r \|_{L^4}
    + \| \partial_r \psi^- \|_{L^2}  \\
    & \les \| \psi^- \|_{\dHe} \approx \| R_- \psi^- \|_{H^1}
  \end{split}
  \]
  which implies that $\| R_+ \psi^+ \|_{H^1} \les \| R_- \psi^-
  \|_{H^1}$.

  If $s=2$ then,
  \[
  \begin{split}
    \partial_r^2 \psi^+
    & = \partial_r  [\frac{2i}r (i A_2 \psi^- - \frac1r (A_2+1)\psi_2)+ \partial_r \psi^-] \\
    & = (\partial_r^2 + \frac2r \partial_r - \frac4{r^2})\psi^- + 2i  \partial_r [\frac{1+A_2}r (i \psi^- - \frac{\psi_2}r) ] \\
    & = (\partial_r^2 + \frac2r \partial_r - \frac4{r^2})\psi^- - 2i \frac{1+A_2}{r^2} (i \psi^- - \frac{\psi_2}r) \\
    & + 2i (\Im(\psi^- \frac{\bar \psi_2}r)+ \frac{|\psi_2|^2}{r^2})(i \psi^- - \frac{\psi_2}r) \\
    & + 2i \frac{1+A_2}r [i \partial_r \psi^- - \frac{1}r (i A_2
    \psi^- - \frac1r (A_2+1)\psi_2 )]
  \end{split}
  \]
  From part ii) of Lemma \ref{LBE} we have that $(\partial_r^2 +
  \frac2r \partial_r - \frac4{r^2}) \psi^- \in L^2$.  All the other
  terms are estimated in $L^2$ by using that $\frac{\psi^-}{r^2} \in
  L^2$, the Sobolev embedding $\psi^- \in L^4 \cap L^6$ which implies
  by \eqref{strfix} that $\frac{\psi_2}{r} \in L^4 \cap L^6$ and that
  $1+A_2=\frac{|\psi_2|^2}{1-A_2}$.

  The term $\frac1r \partial_r \psi^+$ is treated via a similar but
  easier argument.  This shows that $(\partial_r^2+\frac1r \partial_r)
  \psi^+ \in L^2$ which amounts to $R_+ \psi_+ \in L^2$.

  If $s=3$ then it is enough to show that $\partial_r \Delta \psi^+
  \in L^2$ in order to conclude that $R_+ \psi^+ \in \dot H^3$. A
  direct computation gives
  \[
  \begin{split}
    \partial_r (\partial_r+\frac1r) \partial_r \psi^+ = (\partial_r +
    \frac2r) (\partial_r - \frac1r) \partial_r \psi^+ = (\partial_r +
    \frac2r) (\Delta - \frac4{r^2}) \psi^- + (\partial_r + \frac2r) S
  \end{split}
  \]
  where
  \[
  \begin{split}
    \frac1{2i} S = & -2 \frac{1+A_2}{r^2} (i \psi^- - \frac{\psi_2}r)
    + (\Im(\psi^- \frac{\bar \psi_2}r)+ \frac{|\psi_2|^2}{r^2})(i \psi^- - \frac{\psi_2}r) \\
    & + \frac{1+A_2}r [i \partial_r \psi^- - \frac{1}r (i A_2 \psi^- -
    \frac1r (A_2+1)\psi_2 )]
  \end{split}
  \]

  From the hypothesis that $\Delta R_- \psi^- \in H^1$ we obtain that
  $(\Delta-\frac{4}{r^2}) \psi^- \in \dHe$, hence $(\partial_r +
  \frac2r) (\Delta - \frac4{r^2}) \psi^- \in L^2$. The term
  $(\partial_r + \frac2r) S$ is further expanded by using the above
  computation for $S$, the system \eqref{comp1} and the fact that
  $(\partial_r + \frac2r) \frac1{r^2}=0$. By using Sobolev embeddings
  one shows that $(\partial_r + \frac2r) S \in L^2$, the details are
  left to the reader.

\end{proof}

\begin{proof}[Proof of Proposition \ref{recprop}]
  With $\psi_2,A_2$ constructed above, we can reconstruct
  $\psi_1=\psi^- +i\frac{\psi_2}r$. Then we solve the system
  \eqref{return} at the level of $(\bar{u}, \bar{v}, \bar{w})$. We
  would like to solve this system with condition at $\infty$,
  $\bar{u}=-\overrightarrow{k}, \bar{v}=\overrightarrow{i},
  \bar{w}=\overrightarrow{j}$.  But this cannot be done
  apriori. Indeed, consider the coefficient matrix in \eqref{return}
  \[
  M=\left( \begin{array}{ccc} 0 & \Re{\psi_1} & \Im{\psi_1} \\
      -\Re{\psi_1} & 0 & 0 \\ -\Im{\psi_1} & 0 & 0 \end{array} \right)
  \]
  Since $M \notin L^1(dr)$, it is not meaningful to initialize the
  problem \eqref{return} at $\infty$. However $M$ has another structure
  which is a consequence of \eqref{comp} rewritten as $\psi_1= (A_2+1) \psi_1 + i \partial_r \psi_2$.
 Therefore $ M= N+ \partial_r K$ and, by \eqref{2est}, $N,K$ satisfy  
 \[
 \| N \|_{L^1(dr)} + \| K \|_{\dHe} \les \| \psi^- \|_{L^2}
 \]
This inequality localizes on intervals $[r,\infty)$ due to \eqref{loc2est}. 
 This allows us to construct solutions with data at $r=\infty$ by using the iteration
  scheme
  \[
  X=\sum_{i} X_i, \qquad X_0=X(\infty), \qquad X_i(r) =
  \int_{r}^\infty M(s) X_{i-1} ds
  \]
  We run the iteration scheme in the space $C([r,\infty])$ of continuous functions on $(r,\infty)$ which
  have limits at $\infty$. Under the assumption that $X_{i-1} \in C([r,\infty])$ we obtain
\[
\begin{split}
X_i(r)& =\int_{r}^\infty (N(s)+\partial_s K(s)) X_{i-1} ds \\
& = \int_{r}^\infty N(s) X_{i-1} ds- K(r)X_{i-1}(r) - \int_{r}^\infty K(s) \partial_s X_{i-1}(s) ds
\end{split}
\]
and further that
\[
\| \partial_r X_i \|_{L^2([r,\infty))} + \| X_i \|_{C([r,\infty])} 
\les \| \psi^- \|_{L^2([r,\infty))} (\| X_{i-1} \|_{L^\infty([r,\infty])}+ \| \partial_r X_{i-1} \|_{L^2([r,\infty))})
\]
 Therefore, inductively, we obtain
 \[
\| \partial_r X_i \|_{L^2([r,\infty))} + \| X_i \|_{C([r,\infty])} 
\les \| \psi^- \|_{L^2([r,\infty))}^i 
\]
  By choosing $R$ large such that $\| \psi \|_{L^2([R,\infty))}$ is small, 
  we can rely on an iteration scheme to construct the solution $X$
  on $[R,\infty)$. 
  
  The uniqueness of this solution is guaranteed by the conservation
  law $\| X \| = constant$ which follows from the antisymmetry of
  $M$.

  This also guarantees that the orthonormality conditions imposed at
  $\infty$ are preserved (recall that $\infty$,
  $\bar{u}=-\overrightarrow{k}, \bar{v}=\overrightarrow{i},
  \bar{w}=\overrightarrow{j}$).  The solution constructed above can be
  extended to $(0,\infty)$ by running a similar argument on intervals where $\| \psi^- \|_{L^2(I)}$
  is small, where the last interval is of the form $(0,r]$. 
  
  The above argument leads to an estimate of the form
  \[
  \| X \|_{C([0,\infty])} + \| \partial_r X \|_{L^2} \les \| \psi^- \|_{L^2}
  \]
  where by $C([0,\infty])$ we mean continuous functions on $(0,\infty)$ which have limits at $0$ and $\infty$. 
  
  Additional information on $\bar{u},\bar{v},\bar{w}$ will be obtained
  in a different manner. Notice that $\bar{u}_3$ and $\zeta =
  \bar{w}_3- i \bar{v}_3$ solve the system
  \[
  \partial_r \bar{u}_3 = \Im{(\psi_1 \bar{\zeta})}, \qquad \partial_r
  \zeta = i \bar{u}_3 \psi_1
  \]
  which is the same as the one satisfied by $A_2, \psi_2$. Since they obey the same
  conditions at $\infty$ we conclude that $\bar{u}_3=A_2, \zeta = \psi_2$.
  From this and the fact that $|\psi_2|^2+A_2^2=1$ it follows also
  that $|\bar u_1|^2 + |\bar u_2|^2=|\psi_2|^2$. 
  
  Next, we extend the system of vectors to $u,v,w$ using the
  equivariant setup, i.e.  by multiplying them with $e^{\theta
    R}$. Using the identification just described above and the
  orthonormality conditions, it follows that \eqref{return} is
  satisfied for $k=2$.  Therefore we have just established the
  existence of an equivariant map $u$ whose vector field
  $\mathcal{W}^-$ in the gauge $(v,w)$ is $\psi^-$ and whose gauge
  elements are $\psi_1,\psi_2,A_2$. Moreover, we have that
  \[
  E(u)= \pi \| \psi^- \|_{L^2}
  \]

Given two fields $\psi^-, \tilde \psi^-$ we reconstruct $X$ and $\tilde X$ as above. Since the construction is
iterative it also follows that
\[
\| X-\tilde X \|_{C[0,\infty]} + \| \partial_r (X-\tilde X) \|_{L^2} \les \| \psi - \tilde \psi \|_{L^2}
\] 
from estimate for of the derivative part in $E(u-\tilde u)$ follows. Since $u_1 = v_2 w_3 - v_3 w_2$, 
$\tilde u_1 = \tilde v_2 \tilde w_3 - \tilde v_3 \tilde w_2$, $\psi_2 = \bar w_3 - i \bar v_3$ and
$\tilde \psi_2 = \bar{ \tilde w}_3 - i \bar{ \tilde v}_3$ it follows that
\[
\|\frac{u_1 - \tilde u_1}r \|_{L^2} 
\les \| \frac{\psi_2-\tilde \psi_2}r \|_{L^2} \| X \|_{L^\infty} + \| X - \tilde X \|_{L^\infty} \| \frac{\tilde \psi_2}r \|_{L^2}
\les \| \psi^- - \tilde \psi^- \|_{L^2}
\]
A similar argument shows that $\|\frac{u_2 - \tilde u_2}r \|_{L^2}  \les \| \psi^- - \tilde \psi^- \|_{L^2}$ which completes 
the proof of \eqref{udif}. 
 
\end{proof}

\section{The Cauchy problem}

In this section we are concerned with the nonlinear system of
equations \eqref{psieq} which we recall here
\begin{equation*}
  \left\{ \begin{array}{l}
      (i \partial_t  +  H^{-}) \psi^- = (A_0 - 2 \frac{A_2+1}{r^2} 
      - \frac1{r}\Im{(\psi_2 \bar{\psi}^-)}) \psi^-
      \cr
      (i \partial_t  +  H^+) \psi^+ = ( A_0 + 2 \frac{A_2+1}{r^2}
      + \frac1{r}\Im{(\psi_2 \bar{\psi}^+)}) \psi^+
    \end{array} \right.
\end{equation*}
where $\psi_2,A_2,A_0$ are given by \eqref{rela}, \eqref{A2},
respectively \eqref{A0}.  The problem comes with an initial data
$\psi^\pm(t_0)=\psi^\pm_0$ and we would like to understand its
well-posedness on intervals $I \subset \R$ with $t_0 \in I$.

We will be mainly interested in solutions of this system which come
from Schr\"odinger maps with energy below the critical threshold,
i.e. they satisfy the compatibility conditions \eqref{compnew} and
with $\| \psi^+ \|_{L^2}= \| \psi^- \|_{L^2}< 2\sqrt{2}$.

For simplicity we denote the nonlinearities by
\begin{equation} \label{psin}
  \begin{split}
    N^-(\psi^-) & = (A_0 - 2 \frac{A_2+1}{r^2} -
    \frac1{r}\Im{(\psi_2 \bar{\psi}^-)}) \psi^- \\
    N^+(\psi^+) & =( A_0 + 2 \frac{A_2+1}{r^2} + \frac1{r}\Im{(\psi_2
      \bar{\psi}^+)}) \psi^+
  \end{split}
\end{equation}
 
We define the mass of a function $f$ by $M(f):=\| f \|^2_{L^2}$.  The
system \eqref{psieq} formally conserves the mass,
i.e. $M(\psi^-(t))=M(\psi^-(0))$ and $M(\psi^+(t))=M(\psi^+(0))$ for
all $t$ in the interval of existence. Moreover, as discussed in
subsection \ref{COG}, a compatible pair also satisfies $\| \psi^+(0)
\|_{L^2}=\| \psi^-(0) \|_{L^2}$.

\subsection{Strichartz estimates}

We begin our analysis with the linear equation
\begin{equation} \label{beq} (i \partial_t + H_k) u =f, \qquad
  u(0)=u_0
\end{equation}
where we recall $H_k=\partial_r^2 + \frac1r \partial_r -
\frac{k^2}{r^2}$. Note that $H^+=H_0$ and $H^-=H_2$. For consistency
in writing we may also choose to use sometimes $R_+=R_0$ and $R_-=R_2$
(recall \eqref{Ddef}).

Our first claim is that, for each $k$, $u$ satisfies the standard
Strichartz estimates
\begin{equation} \label{STR} \| |\nabla|^s R_k u \|_{L^p_t L^q_r} \les
  \||\nabla|^s R_k u_0 \| + \| |\nabla|^s R_k f \|_{L^{\tilde p'}_t
    L^{\tilde q'}_r}
\end{equation}
where $|\nabla|^s=(-\Delta)^\frac{s}2$ (defined in the usual manner),
$(p,q),(\tilde p, \tilde q)$ are admissible pairs in two dimensions
($\frac{1}{p}+\frac{1}{q}=\frac12, 2 < p \leq \infty$) and $(\tilde
p', \tilde q')$ is the dual pair of $(\tilde p, \tilde q)$. Indeed,
$R_k u$ satisfies the following equation
\[
\begin{split}
  (i\partial_t + \Delta) R_k u= R_k f, \qquad R_k u(0)=R_k u_0
\end{split}
\]
Then \eqref{STR} follows from the Strichartz estimates for the Schr\"odinger equation in
two dimensions.  We need to be able to read the Strichartz estimates
at the level of the radial functions.  For even powers of $s$ we use
the identity $\Delta R_k v = R_k H_k v$, hence
\begin{equation} \label{STRE} \| H_k v \|_{L^p_t L^q_r} = \| \Delta
  R_k v \|_{L^p_t L^q_x}
\end{equation}
and this can be extended to higher regularity but we will not need it.

For odd values of $s$ we use that $|\nabla|^s=|\nabla|
(-\Delta)^{\frac{s-1}2}$ and that for $k \ne 0$
\begin{equation} \label{STRO} \| \partial_r v \|_{L^p_t L^q_r} + \|
  \frac{v}r \|_{L^p_t L^q_r} \les \| |\nabla| R_k v \|_{L^p_tL^q_x}
\end{equation}
while for $k=0$
\begin{equation} \label{STRZ} \| \partial_r v \|_{L^p_t L^q_r} \les \|
  |\nabla| R_k v \|_{L^p_tL^q_x}
\end{equation}

In the context of additional regularity, we need improved
versions of the Strichartz estimates as described below.

\begin{lema} \label{LB} Assume that $u^\pm$ satisfy \eqref{beq} with
  initial data $u_0^\pm$ and forcing $f^\pm$.

  i) If $u_0^+ \in L^2$ is such that $H^+ u_0^+ \in L^2$, then the
  following holds true
  \[
  \||\partial_r^2 u^+| + | \frac{\partial_r u^+}{r}| \|_{L^\infty L^2
    \cap L^4 L^4 \cap L^3L^6} \les \| H^+u_0^+ \|_{L^2} + \| H^+ f^+
  \|_{L^1L^2}
  \]
  ii) If $u_0^- \in L^2$ is such that $H^- u_0^- \in L^2$, then the
  following holds true
  \[
  \||\partial_r^2 u^-| + | \frac{\partial_r u^-}{r}|+
  |\frac{u^-}{r^2}| \|_{L^\infty L^2 \cap L^4 L^4 \cap L^3L^6} \les \|
  H^- u_0^- \|_{L^2} + \| H^- f^- \|_{L^1L^2}
  \]
  iii) If $u_0^- \in L^2$ is such that $H^- u_0^- \in \dHe$, then the
  following holds true
  \[
  \| \frac{1}{r} (\partial_r-\frac1r) \partial_r u^- \|_{L^\infty L^2
    \cap L^4 L^4 \cap L^3L^6} \les \| H^- u_0^- \|_{\dHe} + \| H^- f^-
  \|_{L^1 \dHe}
  \]
\end{lema}
These are improved versions of Strichartz estimates from the following
point of view.  In ii) the inequality for $(\partial_r^2 +
\frac1r \partial_r - \frac{4}{r^2}) u=H_2 u$ is the Strichartz
estimate for $H_2 u$ which follows from \eqref{STR} and \eqref{STRE};
our statement is stronger in saying that each term satisfies the
Strichartz estimate. A similar remark is in place for part i).

\begin{proof} i) The proof follows the same lines as the one in ii),
  though it is easier.

  ii) Without restricting the generality of the argument we can assume
  that $f^-=0$.  Based on the representation formula
  \[
  u^- = \int e^{it\xi^2} J_2(r\xi) \mathcal F_2 u_0^- (\xi) \xi d\xi
  \]
  and by using \eqref{derB}, we compute
  \[
  \partial_r u^- = \frac12 \int e^{it\xi^2} (J_1(r\xi) ) -
  J_3(r\xi))\xi \mathcal F_2 u_0^- (\xi) \xi d\xi =e^{itH_1}g_1 -
  e^{itH_3} g_3
  \]
  with $R_1 g_1, R_3 g_3 \in \dot H^1$. Using \eqref{STR} and
  \eqref{STRO} we obtain the conclusion for $|\partial_r^2 u|$ and
  $|\frac1r \partial_r u|$. Since the estimate holds true for
  $H^-f^-$, it follows for $\frac{|f^-|}{r^2}$.

  iii) We note that, by using \eqref{derB}, $(\partial_r-\frac1r) J_1
  = r \frac1r \partial_r J_1 = -r \frac{J_2}{r^2}= -\frac{J_2}{r}$.
  Using this and again \eqref{derB} we start in part i) and compute
  \[
  \begin{split}
    (\partial_r - \frac1r)\partial_r u^- & = - \frac1{2r} \int e^{it\xi^2} J_2(r\xi) \xi \mathcal F_2 u_0^- (\xi) \xi d\xi \\
    & - \frac14 \int e^{it\xi^2} (J_2(r\xi)-J_4(r\xi)) \xi^2 \mathcal
    F_2 u_0^- (\xi) \xi d\xi
  \end{split}
  \]
  The conclusion follows then from \eqref{STR} and \eqref{STRO}.
\end{proof}

\subsection{Setup and the Cauchy theory}
In order to make estimates shorter, we make the following notation
convention $\| f^\pm \| = \| f^+ \| + \| f^- \|$ for various $f$'s and
$\| \cdot \|$ involved in the rest of the paper.

Since our non-linear analysis relies mostly on the $L^4_{t,r}$ norm,
we define the Strichartz norm of $f:I \times \R^2 \rightarrow \C$ by
$S_I(f):= \| f \|^4_{L^4(I \times \R)}$.  If $t_0 \in I$ then we
define $S_{I,\leq t_0} f = \| 1_{I \cap (-\infty,t_0]} f \|^4_{L^4}$
and $S_{I,\geq t_0} f = \| 1_{I \cap [t_0, \infty)} f \|^4_{L^4}$.

We say that a solution $\psi^\pm: I \times \R \rightarrow \C$ blows up
forward in time if $S_{I,\geq t} \psi^\pm = + \infty, \forall t \in
I$.  Similarly $\psi^\pm$ blows up backward in time if $S_{I,\leq t}
\psi^\pm = + \infty, \forall t \in I$.

A possibility that may occur is that for some interval $I$, $S_{I,\geq
  t_0} \psi^+ = +\infty$ while $S_{I,\geq t_0} \psi^- < \infty$, or
any other combination. However from \eqref{str} it follows that
solutions satisfying the compatibility condition \eqref{compnew} we
have that $S_J(\psi^+) \approx S_J(\psi^-)$ on any time interval
$J$. Therefore for such solutions (which we will be mainly interested
in) the above scenario is ruled out.

Let $\psi_+^\pm \in L^2$. We say that the solution $\psi^\pm: I \times
\R \rightarrow \C$ scatters forward in time to $\psi_+^\pm$ iff $sup
I= + \infty$ and $\lim_{t \rightarrow \infty}
M(\psi^\pm(t)-e^{itH^\pm}\psi_+^\pm)=0$.  We say that the solution
$\psi^\pm: I \times \R \rightarrow \C$ scatters backward in time to
$\psi_{-}^\pm$ iff $inf I= - \infty$ and $\lim_{t \rightarrow -\infty}
M(\psi^\pm(t)-e^{itH^\pm}\psi_-^\pm)=0$.

Our first theorem provides the general Cauchy theory for
\eqref{psieq}.

\begin{theo} \label{CT} Consider the problem \eqref{psieq} (with
  $\psi_2,A_2,A_0$ given by \eqref{rela}, \eqref{A2}, \eqref{A0}) with
  $\psi_0^\pm \in L^2$. Then there exists a unique maximal-lifespan
  solution pair $(\psi^+,\psi^-): I \times \R^2$ with $t_0 \in I$ and
  $\psi^\pm(t_0)=\psi_0^\pm$ with the additional properties:

  i) I is open.

  ii) (Forward scattering) If $\psi^\pm$ do not blow up forward in
  time, then $I_+=[0,\infty)$ and $\psi^\pm$ scatters forward in time
  to $e^{itH^\pm} \psi_+^\pm$ for some $\psi_+^\pm \in L^2$.

  Conversely, if $\psi_+^\pm \in L^2$, then there exists a unique
  maximal-lifespan solution $\psi^\pm$ which scatters forward in time
  to $e^{itH^\pm} \psi_+^\pm$.

  iii) (Backward scattering) A similar statement to ii) holds true for
  the backward in time problem.

  iv) (Small data scattering) There exist $\epsilon > 0$ such that if
  $M(\psi_0^\pm) \leq \epsilon$ then $S_{\R}(\psi^\pm) \les
  M(\psi_0^\pm)$. In particular, the solution does not blow up and
  we have global existence and scattering in both directions.

  v) (Uniformly continuous dependence) For every $A > 0$ and $\epsilon
  > 0$ there is $\delta > 0$ such that if $\psi^\pm$ is a solution
  satisfying $S_J(\psi^\pm) \leq A$ and $t_0 \in J$, and such that
  $M(\psi_0^\pm - \tilde \psi_0^\pm) \leq \delta$, then there exists a
  solution such that $S(\psi^\pm -\tilde \psi^\pm) \leq \epsilon$ and
  $M(\psi^\pm(t)-\tilde \psi^\pm(t)) \leq \epsilon, \forall t \in J$.

  vi) (Stability result) For every $A > 0$ and $\epsilon > 0$ there
  exists $\delta > 0$ such that if $S_J(\psi^\pm) \leq A$, $\psi^\pm$
  approximate \eqref{psieq} in the following sense
  \[
  \| (i \partial_t + H^\pm) \psi^\pm - N^\pm(\psi^\pm) \|_{L^\frac43(J
    \times \R)} \leq \delta,
  \]
  $t_0 \in J, \tilde \psi_0^\pm \in L^2$ and
  $S_J(e^{i(t-t_0)H^\pm}(\psi^\pm(t_0) - \tilde \psi_0^\pm)) \leq
  \delta$, then there exists a solution $\tilde \psi^\pm$ on $I$ to
  \eqref{psieq} with $\tilde \psi^\pm(t_0)= \tilde \psi_0^\pm$ and
  $S_J(\psi^\pm - \tilde \psi^\pm) \leq \epsilon$.

  vii) (Additional regularity) Assume that, in addition, $R_\pm
  \psi^\pm_0 \in H^s$ (recall \eqref{ext}) for $s \in \{1,2,3\}$. If
  $J$ is an interval such that $S_{J}(\psi^\pm) \leq A < + \infty$,
  then the solution $\psi^\pm$ satisfies
  \begin{equation} \label{addreg} \| R_\pm \psi^\pm(t) \|_{H^s} \les_A
    \| R_\pm \psi^\pm_0 \|_{H^s}, \qquad \forall t \in J
  \end{equation}  
  and it has Lipschitz dependence with respect to the initial
  data.
\end{theo}
The above results are concerned with general solutions of
\eqref{psieq}. However, our interest lies in solutions which
correspond to geometric maps.  The next result completes the Cauchy
theory for solutions of \eqref{psieq} which satisfy the compatibility
condition \eqref{compnew}.  The system \eqref{psieq} does not directly
involve the variable $\psi_0$ which is defined in this context by
\eqref{psizero}.

\begin{theo} \label{CT-CC} i) If $\psi^\pm_0 \in L^2$ satisfying the
  compatibility condition \eqref{compnew}, then $\psi^\pm(t)$
  satisfies the compatibility condition \eqref{compnew} for each $t
  \in I$. If, in addition, $R_\pm \psi^\pm_0 \in H^3$ then
  \eqref{compat} and \eqref{curb} are satisfied.
  
  ii) If the solution satisfies the compatibility condition
  \eqref{compnew} and it does not blow up in time then the two
  scattering states (described in ii)) are related by
  \begin{equation} \label{asscat}
    \partial_r r (\psi_+^+ - \psi_+^-) = \psi_+^+ + \psi_+^-
  \end{equation}  
  Conversely, if $\psi_+^\pm \in L^2$ satisfy \eqref{asscat}, then the
  unique maximal-lifespan solution $\psi^\pm$ which scatters to
  $e^{itH^\pm} \psi_+^\pm$ (constructed in part ii)) satisfy the
  compatibility condition \eqref{compnew}.  A similar statement holds
  true for the backward in time scattering.

  iii) If $\psi^\pm$ satisfy the compatibility conditions, then for
  every interval $J \subset I$ (I being the maximal-lifespan interval)
  the following holds true
  \begin{equation} \label{str} \| \psi^+ \|_{L^4(J)} \approx \| \psi^-
    \|_{L^4(J)}
  \end{equation}
  where the constants involved in the use $\approx$ are independent of
  the interval $J$.

\end{theo}

As a consequence of these theorems we are able to prove the following
result
\begin{prop} \label{SMR} If $\psi^\pm_0 \in L^2$ satisfies the
  compatibility conditions \eqref{compnew}, $R_\pm \psi^\pm_0 \in H^2$
  and $\psi^\pm(t)$ is the solution of \eqref{psieq} on $I$ then the
  map $u(t)$ constructed in Proposition \ref{recprop} (for each $t$)
  is a Schr\"odinger map.
\end{prop}

\begin{proof}[Proof of Theorem \ref{CT}]  
  Given the form of the nonlinearities $N^\pm(\psi^\pm)$ and the
  formulas for $\psi_2, A_2, A_0$, we obtain using \eqref{rdrm}
  \[
  \| \frac{\psi_2}r \|_{L^4} \les \| \psi^\pm \|_{L^4}, \quad \|
  \frac{A_2+1}{r^2} \|_{L^2} + \| A_0 \|_{L^2} \les \| \psi^\pm
  \|^2_{L^4}
  \]
  Altogether, these estimates imply that
  \begin{equation} \label{Nest}
  \| N^\pm(\psi^\pm) \|_{L^\frac43} \les \| \psi^\pm \|^3_{L^4}
  \end{equation}
  In a similar fashion, we also obtain that
  \begin{equation} \label{Ndif}
    \begin{split}
      \| N^\pm(\psi^\pm) - N^\pm(\tilde \psi^\pm) \|_{L^\frac43} \les
      \| \psi^\pm - \tilde \psi^\pm \|_{L^4} (\| \psi^\pm \|^2_{L^4} +
      \| \tilde \psi^\pm \|^2_{L^4})
    \end{split}
  \end{equation}
  Based on a standard theory, by using the Strichartz estimates and
  \eqref{Ndif}, the system \eqref{psieq} has a unique solution with
  $\psi^\pm(t_0)=\psi^\pm_0$ on some maximal life-time open interval
  $I$ with $t_0 \in I$.

  If $\| \psi^\pm \|_{L^4(\R_+ \times \R)}$ is finite, then $\psi^\pm$
  scatters at infinity to $e^{itH^\pm} \psi^\pm_+$. Constructing
  solutions to scattering states satisfying \eqref{asscat} is done in
  the usual manner.  Precisely, we start with linearizing near the
  assypmtotic states
  \[
  \psi^- = e^{itH^-} \psi_+^- + e^-, \qquad \psi^+ = e^{itH^+}
  \psi_+^+ + e^+
  \]
  and generate the following system for $e^-,e^+$:
  \begin{equation*}
    \left\{ \begin{array}{l}
        (i \partial_t  +  H^{-}) e^- = (A_0 - 2 \frac{A_2+1}{r^2} 
        - \frac1{r}\Im{(\psi_2 \bar{\psi}^-)}) \psi^-
        \cr
        (i \partial_t  +  H^+) e^+ = ( A_0 + 2 \frac{A_2+1}{r^2}
        + \frac1{r}\Im{(\psi_2 \bar{\psi}^+)}) \psi^+
      \end{array} \right.
  \end{equation*}
  with zero data at $\infty$. By choosing $T$ large enough so that $\|
  e^{itH^-} \psi_+^- \|_{L^4([T,+\infty)} + \| e^{itH^+} \psi_+^+
  \|_{L^4([T,+\infty)} \ll 1$, the solution to this system is obtained
  by a fixed point argument, which ensures existence and uniqueness.
  This finishes the proof of part i)-iii). Part iv)-vi) are standard
  in light of \eqref{Ndif}.

  Parts vii) is usually standard, but our nonlinearity cannot be
  completely reduced to a form where we can invoke the usual
  arguments. We rewrite the nonlinear terms as follows
  \[
  \begin{split}
    A_0 - 2 \frac{A_2+1}{r^2} - \frac1{r}\Im{(\psi_2 \bar{\psi}^-)}
    & = -\frac{|\psi^-|^2}2 +[r\partial_r]^{-1} \Re(\bar \psi^+ \psi^-) + \frac1{2r^2} \int_0^r (|\psi^+|^2 - |\psi^-|^2)sds \\
    A_0 + 2 \frac{A_2+1}{r^2} + \frac1{r}\Im{(\psi_2 \bar{\psi}^+)} &
    = -\frac{|\psi^+|^2}2 +[r\partial_r]^{-1} \Re(\bar \psi^+ \psi^-)
    - \frac1{2r^2} \int_0^r (|\psi^+|^2 - |\psi^-|^2)sds
  \end{split}
  \]
  Without the term $[r\partial_r]^{-1} \Re(\bar \psi^+ \psi^-)$, the
  analysis would be standard.  Indeed, the terms $|\psi^\pm|^2
  \psi^\pm$ can be extended to their two-dimensional regular terms
  $|R_\pm \psi^\pm|^2 R_\pm \psi^\pm$, and for the integral term
  $\frac1{2r^2} \int_0^r (|\psi^+|^2 - |\psi^-|^2)sds$ one notices
  that the operator $\frac1{r^2} \int_0^r \cdot sds$ keeps the two
  dimensional frequency localization; then one could use frequency
  envelopes techniques to deal with the regularity of these terms.
  However, the term $[r\partial_r]^{-1} \Re(\bar \psi^+ \psi^-) $
  cannot be extended to a two-dimensional variant with regular terms
  since $\psi^\pm$ require different phases for completion to regular
  two-dimensional terms.

  We provide a full analysis of the term
  \[
  N_1^\pm = [r\partial_r]^{-1} \Re(\bar \psi^+ \psi^-) \psi^\pm
  \]
  This analysis can be extended to the other two terms in
  $N^\pm(\psi^\pm)$.

  Since $S_I(\psi^\pm) \leq A$, from \eqref{Nest} and \eqref{STR} we obtain
  \[
  \| \psi^\pm \|_{L^3L^6(I \times \R)} \les_A 1
  \]
  Therefore it makes sense to define
  \[
  \begin{split}
    & B=\| \partial_r \psi^\pm \|_{L^3L^6} +   \| \frac{\psi^-}r \|_{L^3L^6}  \\
    &C = \| \partial_r^2 \psi^\pm \|_{L^3 L^6} + \| \frac1r \partial_r
    \psi^\pm \|_{L^3 L^6}+
    \| \frac{\psi^-}{r^2}  \|_{L^3 L^6} \\
    & D = \| \partial_r H^\pm \psi^\pm \|_{L^3L^6} + \| \frac1r H^\pm
    \psi^- \|_{L^3L^6} + \| \frac1r (\partial_r-\frac1r) \partial_r
    \psi^- \|_{L^3L^6}
  \end{split}
  \]
  We prove the following estimates
  \begin{equation} \label{regineq}
    \begin{split}
      & \| \partial_r  N_1^\pm \|_{L^1L^2} + \| \frac1r N_1^- \|_{L^1 L^2}  \les_A B \\
      & \| H^\pm N_1^\pm \|_{L^1 L^2} \les_{A}  C + B^2 \\
      & \| \partial_r H^\pm N_1^\pm \|_{L^1L^2} + \| \frac1r H^- N_1^-
      \|_{L^1 L^2} \les_A D + BC
    \end{split}
  \end{equation}
  Similar estimates can be established for the other two terms
  in $N^\pm(\psi^\pm)$, and when combined with the results in
  Lemma \ref{LB}, a standard argument establishes the conclusion in
  \eqref{addreg}.

  We now turn to the proof of \eqref{regineq}. We compute
  \[
  \partial_r N_1^\pm = \partial_r \left( [r\partial_r]^{-1} \Re(\bar
    \psi^+ \psi^-) \right) \psi^\pm + [r\partial_r]^{-1} \Re(\bar
  \psi^+ \psi^-) \partial_r \psi^\pm
  \]
  and estimate
  \[
  \| \partial_r N_1^\pm \|_{L^1L^2} \les \| \psi^+ \|_{L^3L^6} \|
  \frac{ \psi^-}r \|_{L^3 L^6} \| \psi^\pm \|_{L^3L^6} + \| \psi^\pm
  \|^2_{L^3L^6} \| \partial_r \psi^\pm \|_{L^3L^6}
  \]
  from which half of the first estimate in \eqref{regineq} follows;
  the second half follows in a similar manner.

  We continue with
  \[
  \begin{split}
    H^\pm N_1^\pm = & \Delta \left( [r\partial_r]^{-1} \Re(\bar \psi^+
      \psi^-) \right) \psi^\pm
    + 2 \partial_r \left( [r\partial_r]^{-1} \Re(\bar \psi^+ \psi^-) \right) \partial_r \psi^\pm \\
    & + \left( [r\partial_r]^{-1} \Re(\bar \psi^+ \psi^-) \right)
    H^\pm \psi^\pm
  \end{split}
  \]
  The last term is estimated by $\les_A C$, the second one is
  estimated by $\les_A B^2$, while the first one equals
  \[
  (\partial_r + \frac1r) \frac{\Re(\bar \psi^+ \psi^-)}r \cdot
  \psi^\pm = \frac{\Re(\partial_r \bar \psi^+ \cdot \psi^-)+ \Re(\bar
    \psi^+ \cdot \partial_r \psi^-)}r \cdot \psi^\pm
  \]
  and its $L^1L^2$ norm is estimated by
  \[
  \les ( \| \partial_r \psi^+ \|_{L^3L^6} \| \frac{\psi^-}r
  \|_{L^3L^6}+ \| \psi^+ \|_{L^3L^6} \| \frac{\partial_r \psi^-}r
  \|_{L^3L^6} ) \| \psi^\pm \|_{L^3L^6}
  \]
  from which the second estimate in \eqref{regineq} follows.

  For the third estimate we start with
  \[
  \begin{split}
    \partial_r H^\pm N_1^\pm & = \partial_r \Delta \left(
      [r\partial_r]^{-1} \Re(\bar \psi^+ \psi^-) \right) \psi^\pm
    + \Delta \left( [r\partial_r]^{-1} \Re(\bar \psi^+ \psi^-) \right) \partial_r \psi^\pm \\
    & + 2 \partial_r^2 \left( [r\partial_r]^{-1} \Re(\bar \psi^+
      \psi^-) \right) \partial_r \psi^\pm
    + 2  \partial_r \left( [r\partial_r]^{-1} \Re(\bar \psi^+ \psi^-) \right) \partial_r^2 \psi^\pm \\
    & + \partial_r \left( [r\partial_r]^{-1} \Re(\bar \psi^+ \psi^-)
    \right) H^\pm \psi^\pm + \left( [r\partial_r]^{-1} \Re(\bar \psi^+
      \psi^-) \right) \partial_r H^\pm \psi^\pm
  \end{split}
  \]
  The $L^1L^2$ norm of the sixth terms above is bounded by $\les_A
  D$. Using the previous arguments, the $L^1L^2$ norm of the second,
  fourth and fifth term is bounded by $\les_A BC$.  Since
  \[
  \begin{split}
    \partial_r^2 \left( [r\partial_r]^{-1} \Re(\bar \psi^+ \psi^-)
    \right) = \frac{\Re(\partial_r \bar \psi^+ \cdot \psi^-)+ \Re(\bar
      \psi^+ \cdot \partial_r \psi^-)}r - \frac{\Re(\bar \psi^+
      \psi^-)}{r^2}
  \end{split}
  \]
  it follows that the $L^1L^2$ norm of the third term above is bounded
  by $\les_A BC$.

  Finally we compute
  \[
  \begin{split}
    & \partial_r \Delta \left( [r\partial_r]^{-1} \Re(\bar \psi^+
      \psi^-) \right)
    = \partial_r \left( \frac{\Re(\partial_r \bar \psi^+ \cdot \psi^-)+ \Re(\bar \psi^+ \cdot \partial_r \psi^-)}r \right) \\
    & = \frac{\Re(\partial_r^2 \bar \psi^+ \cdot \psi^-)+ 2
      \Re(\partial_r \bar \psi^+ \cdot \partial_r \psi^-)}r + \Re(\bar
    \psi^+ \frac1r(\partial_r - \frac1r) \partial_r \psi^-)
  \end{split}
  \]
  and this allows us to estimate the first term by $BC+D$. This
  finishes the argument for \eqref{regineq}.
\end{proof}

\begin{proof}[Proof of Theorem \ref{CT-CC}]
  i) It is useful to rephrase this in terms of $\psi_1$, $\psi_2$,
  which are recovered linearly from $\psi^{\pm}$. Reverting the
  algebraic computation from Sections \ref{CG} and \ref{COG},
  $\psi_1$, $\psi_2$ solve the system \eqref{dtpsiab}. Then we seek to
  show that the relation $D_1 \psi_2 = D_2 \psi_1$ is preserved along
  the flow. For this we will derive an equation for the quantity
  \[
  F = D_2 \psi_1 - D_1 \psi_2
  \]
  It is convenient to use covariant notations. Given $\psi_1$,
  $\psi_2$ and $A_1=0$, $A_2$, we define $\psi_0$ via
  \eqref{psizero}. Then we consider the connection $(A_0,A_1,A_2)$ and
  identify its curvature.  The equation \eqref{A2} reads
  \[
  \partial_1 A_2 - \partial_2 A_1 = \Im(\psi_1 \bar \psi_2)
  \]
  The equation \eqref{A0} yields \eqref{a0}, which in turn gives
  \[
  \partial_1 A_0 - \partial_0 A_1 = \Im(\psi_1 \bar \psi_0) -
  \frac{1}{r^2}\Re ( F \bar \psi_2)
  \]
  For the remaining curvature component we use \eqref{psieq} to
  compute
  \[
  \begin{split}
    \partial_0 A_2 (r) = & \ \frac12 \int_{r}^\infty \Re ( \psi^+ \bar
    \psi^+_t) - \Re ( \psi^- \bar \psi^-_t) sds
    \\
    = & \ \frac12 \int_{r}^\infty \Im ( \psi^+ \Delta \bar \psi^+) -
    \Im ( \psi^- \Delta \bar \psi^-) sds \\ = & \ \frac12
    \int_{r}^\infty \partial_s \Im ( s \psi^+ \partial_s \bar \psi^+)
    -
    \partial_s \Im (s \psi^- \partial_s \bar \psi^-) ds \\ = & \ -
    \frac{r}2\left(\Im ( \psi^+ \partial_r \bar \psi^+) - \Im (
      \psi^- \partial_r \bar \psi^-)\right) \\ = & \ r \Re( \partial_r
    \psi_1 \frac{\bar \psi_2}r - \psi_1 \partial_r \frac{\bar
      \psi_2}r)
  \end{split}
  \]
  On the other hand we expect to have
  \[
  \partial_0 A_2 (r) = \Im ( \psi_0 \bar \psi_2) = \Re
  ((\partial_r+\frac1r) \psi_1 + i \frac{A_2}{r^2} \psi_2)\bar \psi_2)
  = \Re ((\partial_r+\frac1r) \psi_1\bar \psi_2)
  \]
  Taking the difference, what we actually get is
  \[
  \partial_0 A_2 - \partial_2 A_0 = \Im ( \psi_0 \bar \psi_2) - \Re [F
  \bar \psi_1]
  \]
  The next step is to identify the remaining two quantities in the
  compatibility conditions. We claim that
  \[
  D_1 \psi_0 - D_0 \psi_1 = - \frac{i}{r^2} D_2 F
  \]
  \[
  D_2 \psi_0 - D_0 \psi_2 = i (D_1+\frac{1}r) F
  \]
  For this it suffices to compare the two equations in \eqref{psiab},
  derived as before from \eqref{psizero}, with the equations
  \eqref{dtpsiab} obtained algebraically from \eqref{psieq}.  Now we
  can compute
  \[
  \begin{split}
    D_0 F = & \ D_0 D_2 \psi_1 - D_0 D_1 \psi_2
    \\
    = & \ D_2 D_0 \psi_1 - D_1 D_0 \psi_2 + i \Im ( \psi_0 \bar
    \psi_2) \psi_1 - i \Re (F \bar \psi_1) \psi_1 \\ & \ + i
    \Im(\psi_1 \bar \psi_0)\psi_2 - i \frac{1}{r^2} \Re ( F \bar
    \psi_2) \psi_2
    \\
    = & \ D_2 D_1 \psi_0 - D_1 D_2 \psi_0 + i(\frac{1}{r^2} D_2 D_2 +
    D_1 (D_1+\frac{1}r))F \\ & \ + i \Im ( \psi_0 \bar \psi_2) \psi_1
    - i \Re (F \bar \psi_1) \psi_1 + i \Im(\psi_1 \bar \psi_0)\psi_2 -
    i \Re (
    F \bar \psi_2) \psi_2 \\
    = & \ i(\frac{1}{r^2} D_2 D_2 + D_1 (D_1+\frac{1}r))F - i \Re (F
    \bar \psi_1) \psi_1- i \frac{1}{r^2} \Re ( F \bar \psi_2)\psi_2
    \\
    & \ + i \Im(\psi_1 \bar \psi_1)\psi_0 + i \Im ( \psi_0 \bar
    \psi_2) \psi_1 +i \Im(\psi_1 \bar \psi_0)\psi_2
  \end{split}
  \]
  The last line algebraically vanishes, so we have our equation for
  $F$:
  \[
  i D_0 F = (\frac{A_2^2}{r^2} - \partial_1 (\partial_1+\frac{1}r))F +
  \Re (F \bar \psi_1) \psi_1 + \frac{1}{r^2} \Re ( F \bar
  \psi_2)\psi_2
  \]
  It is more convenient to recast this as an equation for
  \[
  \frac{F}{r} = (\partial_r + \frac1r)\frac{\psi_2}{r} - \frac{i
    A_2}{r} \psi_1
  \]
  which is exactly the quantity in \eqref{compnew}. We obtain
  \begin{equation} \label{Freq} (i \partial_t +\Delta - \frac{1}{r^2})
    \frac{F}{r} = (A_0+\frac{A_2^2-1}{r^2}) F + \Re (\frac{F}{r} \bar
    \psi_1) \psi_1 +\frac{1}{r^2} \Re ( \frac{F}{r} \bar \psi_2)\psi_2
  \end{equation}
  In view of the $L^4$ Strichartz bounds for $\psi_1$ and $\psi_2$ and
  the derived $L^2$ bounds for $A_0$ and $\frac{A_2^2-1}{r^2}$,
  standard arguments show that this linear equation is well-posed in
  $L^2$. Hence the conclusion follows provided that $\frac{F}{r}$ has
  sufficient regularity. Indeed, we have
  \[
  \frac{F}r = -\frac{i}2 \left( \partial_r \psi^+ + \frac{1+A_2}r
    \psi^+ - \partial_r \psi^- - \frac{1-A_2}r \psi^- \right)
  \]
  It is obvious that if $R_\pm \psi^\pm \in H^1$ then $\frac{F}r \in
  L^2$.
  
  If $R_\pm \psi^\pm \in H^2$ then by using the results in Lemma
  \ref{LBE} and Sobolev embeddings one easily shows that $\frac{F}r
  \in \dHe$.
  
  We will show in detail that if $R_\pm \psi^\pm \in H^3$, then $ (
  \Delta - \frac1{r^2}) \frac{F}r \in L^2$. Indeed,
  \[
  \begin{split}
    2i ( \Delta - \frac1{r^2}) \frac{F}r & = \partial_r \Delta \psi^+
    + \frac{1+A_2}r \Delta \psi^+ + \frac{\Delta A_2}{r} \psi^+
    + 2 (\frac{\partial_r A_2}r  - \frac{1+A_2}{r^2}) \partial_r \psi^+ - 2 \frac{\partial_r A_2}{r^2} \psi^+ \\
    &- (\partial_r + \frac2r)( \Delta - \frac{4}{r^2} )\psi^- +
    \frac{\Delta A_2}{r} \psi^- + 2 (\frac{\partial_r A_2}r -
    \frac{1+A_2}{r^2}) \partial_r \psi^- - 2 \frac{\partial_r
      A_2}{r^2} \psi^-
  \end{split}
  \]
  We have that $\| \partial_r \Delta \psi^+ \|_{L^2} \les \| R_+
  \psi^+ \|_{H^3}$. For some terms, direct computations combined with
  Sobolev embeddings give
  \[
  \| \frac{1+A_2}r \Delta \psi^+ \|_{L^2} \les \| \psi^\pm \|_{L^4}^2
  \| \Delta \psi^+ \|_{H^2} \les \| R_\pm \psi^\pm \|_{H^2}^3
  \]
  \[
  \| ( \frac{\partial_r A_2}{r}- \frac{1+A_2}{r^2} ) \partial_r
  \psi^+\|_{L^2} \les \| \psi^\pm \|_{L^\infty}^2 \| \partial_r \psi^+
  \|_{L^2} \les \| R_\pm \psi^\pm \|_{H^2}^3
  \]
  while for others, we rearrange them
  \[
  \frac1r (\Delta A_2 - 2 \frac{\partial_r A_2}r) \psi^+ =
  \frac{1}{4r} (\partial_r - \frac1r) [r(|\psi^-|^2 - |\psi^+|^2)]
  \psi^+ = \frac14 \partial_r (|\psi^-|^2 - |\psi^+|^2) \psi^+
  \]
  and then estimate by
  \[
  \| \frac1r (\Delta A_2 - 2 \frac{\partial_r A_2}r) \psi^+ \|_{L^2}
  \les \| \psi^\pm \|^2_{L^\infty} \| \partial_r \psi^\pm \|_{L^2}
  \les \| R_\pm \psi^\pm \|_{H^2}^3
  \]
  The terms corresponding to $\psi^-$ are treated in a similar manner.
  
  Hence we can conclude that $(\Delta - \frac1{r^2}) \frac{F}r \in
  L^2$.  This allows us to run a standard energy argument by pairing
  the equation, with $\bar F$, to conclude that
  \[
  \partial_t \| \frac{F}r \|^2_{L^2} \les (\| \psi_1 \|^2_{L^\infty} +
  \| \frac{\psi_2}r \|^2_{L^\infty}) \| \frac{F}r \|^2_{L^2}
  \]
  which by using the Gronwall inequality and the fact that $F(0)=0$
  leads to $F(t)=0$ for all $t \in I$.
  
  In order to run the energy argument it suffices to have $\frac{F}r
  \in \dHe$ and use the pairing of $\dot H^{-1}_e$ and $\dHe$. This is
  useful in the proof of Proposition \ref{SMR} where we assume only
  $R_\pm \psi^\pm \in H^2$.
  
  In the general case when $\psi^\pm_0 \in L^2$ only, we regularize
  them as follows. We produce $R_\pm \psi^-_{n,0} \in H^3$ so that $\|
  \psi^-_0 - \psi^-_{n,0} \|_{L^2} \leq \frac1n$.  By using Lemma
  \eqref{a2p2c}, and particularly part v), we obtain that the
  compatible pair $R_+ \psi^+_{n,0} \in H^3$ and $\| \psi^+_0 -
  \psi^+_{n,0} \|_{L^2} \leq \frac1n$. We also recast the
  compatibility condition to
  \[
  \psi^+-\psi^-=- [r\partial_r]^{-1} \left( \psi^+ - \psi^-
    +A_2(\psi^++\psi^-)\right)
  \]
  so that all terms involved belong to $L^2$.  Using the conservation
  of the compatibility condition for $\psi^\pm_n(t)$ under the flow
  \eqref{psieq} and part v) of the Theorem, we obtain the desired
  result.

  ii) The key observation is that the equation for $\psi_2$ in
  \eqref{comp1} becomes linear in the following sense:
  \begin{equation} \label{limcomp} \lim_{t \rightarrow \infty}
    \| \partial_r \psi_2 + i \psi^- - \frac{\psi_2}r \|_{L^2} =0
  \end{equation}
  under the hypothesis that $\lim_{t \rightarrow \infty} \| \psi^-(t)
  - e^{itH^-} \psi^-_+ \|_{L^2}=0$. This is easily shown to follow
  from the following estimate
  \begin{equation} \label{pd} \lim_{t \rightarrow \infty} \sup_{r \in
      (0,\infty)} | r \int_r^\infty \frac{e^{itH^-} f}{s^2} ds| = 0
  \end{equation}
  which holds true for $f \in L^2$ and will be proved in the
  Appendix. Based on this, it follows that $\lim_{t \rightarrow
    \infty} \| \psi_2(t) \|_{L^\infty} =0$, and that
  \[
  \lim_{t \rightarrow \infty} \|i (A_2+1) \psi^- - \frac{1}r (A_2+1)
  \psi_2\|_{L^2} =0
  \]
  which justifies \eqref{limcomp}.

  With the notation
  \[
  f(t) = \partial_r ( e^{itH^+} \psi_+^+ - e^{itH^-}\psi_+^-) - 2
  \frac{e^{itH^-}\psi_+^-}r,
  \]
  the scattering relation \eqref{asscat} can be rewritten as $\lim_{t
    \rightarrow \infty} \| f(t) \|_{\dot H_e^{-1}}=0$.  A direct
  computation gives that $f$ obeys the equation
  \[
  \begin{split}
    (i\partial_t + \Delta) f & = [\Delta, \partial_r]( e^{itH^+}
    \psi_+^+ - e^{itH^-}\psi_+^-) - 2 [\Delta, \frac1r]
    e^{itH^-}\psi_+^-
    - \frac{8}{r^2} \frac{e^{itH^-}\psi_+^-}r - \partial_r \frac{4}{r^2} e^{itH^-}\psi_+^- \\
    & = \frac{1}{r^2} \partial_r e^{it\Delta} \psi_+^+ -2( -2 \frac{\partial_r}{r^2} + \frac1{r^3}) e^{itH^-}\psi_+^-  - \frac{1}{r^2} \partial_r e^{itH^-}\psi_+^- - \frac{4}{r^2} \partial_r e^{itH^-}\psi_+^- \\
    & = \frac{1}{r^2} f(t)
  \end{split}
  \]
  Since $\lim_{t \rightarrow \infty} \| f(t) \|_{\dot H^{-1}_e} =0$ it
  follows from the conservation of the $\dot H^{-1}_e$ norm that
  $f(0)=0$ which is \eqref{asscat}. Alternatively, one could carry out
  this argument as we did in viii).

  Assume now that given $\psi^\pm_+$ satisfying \eqref{asscat} we
  construct (as in ii)) solutions $\psi^\pm(t)$ to \eqref{psieq} on
  some $[T,+\infty)$ which scatter forward to $e^{itH^\pm}
  \psi^\pm_+$. Following the argument in part i) we construct $F$
  which satisfies \eqref{Freq}. Assuming additional regularity on the
  states $\psi^\pm_+$, $R_\pm \psi^\pm_+ \in H^3$, we have by part
  vii) of Theorem \ref{CT} that $R_\pm \psi^\pm(t) \in H^3$, hence by
  the argument in i) $(\Delta - \frac1{r^2})\frac{F}r \in L^2$ and the
  right-hand side of \eqref{Freq} belongs to $L^2$. Then the Duhamel
  formula applies to \eqref{Freq} and in turn the Strichartz estimate
  \[
  \| \frac{F}r \|_{L^4([T,\infty) \times \R)} \les \| \psi^\pm
  \|_{L^4([T,\infty) \times \R)}^2 \| \frac{F}r \|_{L^4([T,\infty)
    \times \R)}
  \]
  where we have used that $\lim_{t \rightarrow \infty} \| \frac{F(t)}r
  \|_{L^2}=0$ (this follows as above because of \eqref{pd}). Next, by
  taking $T$ large enough, we obtain that $F(t) \equiv 0$ for $t \geq
  T$ and the conclusion follows by invoking part i).

  For general states $\psi^\pm_+ \in L^2$ satisfying \eqref{asscat} we
  proceed as above.  We approximate them by sequences $\psi^\pm_{n,+}$
  with $R_\pm \psi^\pm_{n,+} \in H^3$; this can be done by
  regularizing $R_- \psi^-_+$ first and then showing that the
  corresponding $R_+ \psi^+_+$ has the same regularity as we did in
  Lemma \ref{a2p2c} part v) - in fact this argument involves only the
  linear part of the argument there. Then we write \eqref{asscat} at
  the level of $L^2$
  \[
  \psi^+ - \psi^-= 2 [r\partial_r]^{-1} \psi^-,
  \]
  use the above argument and a limiting argument.

  iii) One side of \eqref{str} follows from the fixed time bound
  \eqref{strfix}. The other side is similar, and it consists in
  replicating the result of Lemma \ref{a2p2c} starting from $\psi^+$
  instead.

\end{proof}

\begin{proof}[Proof of Proposition \ref{SMR}]
  With the given $\psi_0^\pm$ we reconstruct $u_0 \in \dot H^1 \cap
  \dot H^3$ as in Proposition \ref{recprop}.  The additional
  regularity $R_\pm \psi^\pm_0 \in H^2$ implies, by \eqref{rtr}, that
  $u_0 \in \dot H^1 \cap \dot H^3$.  For the classical Schr\"odinger
  Map $u(t)$ with data $u_0$ (whose existence follows from Theorem \ref{clasic}) 
  we construct its Coulomb gauge, its field
  components and write the system \eqref{psieq} whose initial data is
  $\psi_0^\pm$. Invoking the uniqueness part of Theorem \ref{CT}, it
  follows that $\psi^\pm(t)$ are the gauge representation of $\mathcal
  W^\pm(t)$, hence the reconstruction in Proposition \ref{recprop}
  gives the Schr\"odinger Map $u(t)$ for each $t$.

\end{proof}

We can now identify the critical threshold for global well-posedness
and scattering. For any $m \geq 0$, we define $A(m)$ by
\[
A(m):= sup \{ S_{I_{max}}(\psi^-): M(\psi^\pm) \leq m \ \mbox{where
  $\psi^\pm$ is a solution to \eqref{psieq} satisfying}
\eqref{compnew} \}
\]
where $\psi^\pm$ is assumed to be a solution of \eqref{psieq},
satisfying the compatibility condition \eqref{compnew} and $I_{max}$
is its maximal interval of existence. Note that the compatibility
condition forces $M(\psi^+)=M(\psi^-)$ so the use of $M(\psi^\pm) \leq
m$ is unambiguous.

Obviously $A$ is a monotone increasing functions, it is bounded for
small $m$ by part iv) and it is left-continuous by part v) of Theorem
\ref{CT}. Therefore there exists a critical mass $0 < m_0 \leq +
\infty$ such that $A(m)$ is finite for all $m < m_0$ and it is
infinite $m \geq m_0$. Also any solution $\psi^\pm$ with $M(\psi^\pm) < m_0$
is globally defined and scatters.

Note that from \eqref{str} it follows that we could have used
$S_{I_{max}}(\psi^+)$ in the definition of $A(m)$ and arrive to the
same conclusion as above with the same critical mass $m_0$.

\section{Concentration compactness} \label{concomp} We start by
exhibiting the symmetries of the system \eqref{psieq}.  First, the system is
invariant under the time reversal transformation $\psi^\pm(r,t)
\rightarrow \bar \psi^\pm(r,-t)$.  This allows us to focus our attention on
positive times, i.e. $t \geq 0$.

Next, the system is invariant under two other transformations:
scaling, $\psi^\lambda = \lambda^{-1} \psi(\lambda^{-1} r,
\lambda^{-2} t)$ with $\lambda \in \R$, and phase multiplication,
$\psi^\alpha(r,t) = e^{i \alpha} \psi(r,t)$ with $\alpha \in \R/{2\pi
  \Z}$.  The phase multiplication can be ignored as the group
generated is compact. This way we generate the first (non-compact)
group $G$ of transformations $g_\lambda$ defined by
\[
g_{\lambda} f(r) = \lambda^{-1} f(\lambda^{-1} r)
\]
From \eqref{rela}, \eqref{A2} and \eqref{A0}, the effect of the action
$g_\lambda$ on $\psi^\pm$ is translated in the action of
$g_{\lambda}^1$ on $\psi_2,A_2$ and $g_\lambda^2$ on $A_0$ where
\[
g^1_{\lambda} f(r) = f(\lambda^{-1} r), \qquad g^2_{\lambda} f(r) =
\lambda^{-2} f(\lambda^{-1} r)
\]
The action of $g_\lambda$ is extended to space-time functions by
\[
T_{g_\lambda} f(r,t) = \lambda^{-1} f(\lambda^{-1} r, \lambda^{-2} t)
\]

The equations in \eqref{psieq} are also time translation invariant and
this suggests enlarging the group $G$ to $G'$ as follows. Given
$\lambda > 0$ and $t_0 \in \R$, we define
\[
g_{\lambda, t_0} f = \lambda^{-1} [e^{it_0H^-} f](\lambda^{-1} r)
\]
We denote by $G'$ the groups generated by these transformations. We also define
the extend the action of $g_{\lambda,t_0}$ to space-time functions by  
\[
T_{g_{\lambda,t_0}} f = T_{g_\lambda} f(\cdot, \cdot + t_0)
\]

Given two sequences $g^{n}, \tilde{g}^{n} \in G'$, we say
that they are asymptotically orthogonal iff
\begin{equation} \label{assort} \frac{\lambda_n}{\tilde \lambda_n} +
  \frac{\tilde \lambda_n}{\lambda_n} + | t_n \lambda_n^2 - \tilde t_n
  \tilde \lambda_n^2|= \infty
\end{equation}
The asymptotic orthogonality will be mostly exploited as
follows. Given two asymptotically orthogonal sequences $g^{n},
\tilde{g}^{n} \in G'$
\begin{equation} \label{assmas} \lim_{n \rightarrow \infty} \lan g^{n}
  f, \tilde g^{n} \tilde f \ran = 0, \quad \forall f, \tilde f \in L^2
\end{equation}
\begin{equation} \label{assstr} \lim_{n \rightarrow \infty} \|
  |T_{g^n} h|^\frac12 |T_{\tilde g^{n}} \tilde h|^\frac12 \|_{L^4} =
  0, \quad \forall h, \tilde h \in L^4(\R \times \R).
\end{equation}
\begin{equation} \label{assrdrm} \lim_{n \rightarrow \infty} \|
  T_{g^{n}} f \cdot [r\partial_r]^{-1} \left( T_{\tilde g^{n}} h
    T_{\tilde g^{n}} \tilde h \right) \|_{L^\frac43} = 0, \quad
  \forall f, h, \tilde h \in L^4(\R \times \R).
\end{equation}
\begin{equation} \label{assint} \lim_{n \rightarrow \infty} \|
  T_{g^{n}} f \cdot r^{-2}[r^{-1} \bar \partial_r]^{-1} \left(
    \int_0^r T_{\tilde g^{n}} h T_{\tilde g^{n}} \tilde h \right)
  \|_{L^\frac43} = 0, \quad \forall f, h, \tilde h \in L^4(\R \times
  \R).
\end{equation}

These estimate are inspired by the similar ones in \cite{BG},
\cite{Ker}, \cite{mv}, \cite{tvz}, \cite{bv} to which we also refer
the reader for further details on the subject.  They can also be
proved directly using the Riemann-Lebesgue characterization of $L^p$
spaces.

We are now ready to state the two main results of this section.

\begin{theo} \label{cc} Assume that the critical mass $m_0 < 8$. Then
  there exists a critical element, i.e. a maximal-lifespan solution
  $\psi^\pm$ to \eqref{psieq} and satisfying \eqref{compnew}, with
  mass $m_0$ which blows up forward in time. In addition this solution
  has the following compactness property: there exists a continuous
  function $\lambda(t): I_+=[0,T_+) \rightarrow \R_+$ such that the
  sets
  \[
  K^\pm:= \left\{ \frac1{\lambda(t)} \psi^\pm(\frac{r}{\lambda(t)},t),
    t \in I_+ \right\}
  \]
  are precompact in $L^2$.
\end{theo}

\begin{rema} As a consequence of the compactness property it follows
  that there exists a function $C: \R^+ \rightarrow \R^+$ such that
  the above critical element satisfies
  \begin{equation} \label{psicon} \int_{r \geq C(\eta)
      \lambda(t)^{-1}} |\psi^\pm(t,r)|^2 rdr \leq \eta, \quad \forall
    t \in I_+.
  \end{equation}  

\end{rema}

One can construct critical elements whose function $\lambda(t)$ has
more explicit behavior.
\begin{theo} \label{lambdacontrol} Assume that the critical mass $m_0
  < 8$. Then we can construct a critical element as in Theorem
  \ref{cc} such that one of the two scenarios holds true:

  i) $T_{+}=\infty$ and $\lambda(t) \geq c > 0, \forall t \geq 0$.

  ii) $T_+ < \infty$ and $\lim_{t \rightarrow T_+} \lambda(t)=\infty$.
\end{theo}

The rest of this section is spent on sketching the proof of the two
theorems above.  The statements of Theorema \ref{cc} and
\ref{lambdacontrol} are similar to the corresponding ones in the work
of Kenig and Merle, see \cite{KeMe1}.

It is standard, see for instance \cite{KeMe1} and \cite{tvz} that the
result in Theorem \ref{cc} follows from the following
\begin{prop} \label{pcc} Assume $m_0 < 8$. Let $\psi_n^\pm:
  I_{n+}=[0,T_{n+}) \times \R \rightarrow \C, n \in \N$ be a sequence
  of solutions to \eqref{psieq}, satisfying \eqref{compnew} and such
  that $\lim_{n \rightarrow \infty} M(\psi^\pm_n)=m_0$ and $\lim_{n
    \rightarrow \infty} S_{I_{n+}}(\psi_n^\pm) = \infty$.  Then there
  are group elements $g_n \in G$ such that the sequence $g_n
  \psi_n^\pm(t_n)$ has a subsequence which converges in $L^2$.
\end{prop}

One of the main ingredients in the proof of Proposition \ref{pcc} is
the classical linear profile decomposition result.  These type of
results originate in the work of Bahouri and Gerard \cite{BG}, for the case of 
nonlinear wave equation and independently, and in the work of Merle and
Vega \cite{mv}, for the case of the nonlinear Schr\"odinger equation. For the case of
nonlinear Schr\"odinger equations see also \cite{bv}, \cite{Ker}, \cite{tvz}.

\begin{prop} \label{lprof} Let $\psi_0^n, n\in \N$ a bounded sequence
  in $L^2$. Then (after passing to a subsequence if necessary) there
  exists a sequence $\phi^j, j \in \N$ of functions in $L^2$ and
  $g^{n,j} \in G', n,j \in \N$ such that we have the decomposition
  \begin{equation}
    \psi_0^n = \sum_{j=1}^l g^{n,j} \phi^j + w^{n,l}, \qquad \forall l \in \N
  \end{equation}
  where $w^{n,l}$ satisfies
  \begin{equation} \label{strvanish} \lim_{l \rightarrow \infty}
    \lim_{n \rightarrow \infty} S(e^{itH^-} w^{n,l}) = 0
  \end{equation}
  Moreover $g^{n,j}$ and $g^{n,j'}$ are asymptotically orthogonal for
  any $j \ne j'$ and we have the following orthogonality condition
  \begin{equation} \label{weak}
    \begin{split}
      weak \lim_{n \rightarrow \infty} (g^{n,j})^{-1} w^{n,l} =0,
      \quad \forall 1 \leq j \leq l
    \end{split}
  \end{equation}
  As a consequence the mass decoupling property holds
  \begin{equation} \label{massdec} \lim_{n \rightarrow \infty} (M(u^n)
    - \sum_{j=1}^l M(\phi^j)-M(w^{n,l}))=0
  \end{equation}
\end{prop}
The same statement holds true also for the operator $H^+$.  The
general version of this result was established originally in \cite{mv}
and it has the same statement with the following adjustments: the
sequence of functions $\psi^n_0 \in L^2(\R^2)$ and the group $G$ is
enlarged to capture the additional symmetries available in the general
case, i.e. spatial and frequency translation, and the Galilean
transformation. Our statement is closer in spirit to the one in
\cite{tvz}, see Theorem 7.3, where the statement is restricted to the
radial case.  Our statement is the corresponding one for the case of
"equivariant" functions, i.e.  functions with the property that
$u(r,\theta)=e^{2i\theta} v (r)$ (in the case of $H^+$ the result is
needed for radial function and then the Theorem 7.3 in \cite{tvz}
applies verbatim).  Indeed, as we already noticed before, the
transformation $R_- e^{itH^-} \psi^- = e^{it\Delta} (R_- \psi^-)$
takes one-dimensional homogeneous solutions of \eqref{beq} into
two-dimensional homogeneous "equivariant "solutions of the free
Schr\"odinger equation. In \cite{tvz}, see the proof of Theorem 7.3, a
robust argument derives the radial statement from the general case.
The argument there can be easily replicated for the "equivariant"
solutions.

\begin{proof}[Proof of Proposition \ref{pcc}] A key ingredient in this
  proof is to produce a geometric profile decomposition in the spirit
  of Proposition \ref{lprof}. The challenge is that the profiles
  cannot be arbitrary, but geometrically related via the compatibility
  condition \eqref{compnew}.

  One choice is to produce the profile decomposition for
  $\psi_n^\pm(0)$ and then perform additional analysis on the profiles
  so that they satisfy the compatibility relations \eqref{compnew}.
  This route had been implemented in the case of Wave Maps by Krieger
  and Schlag in \cite{KS}.

  We choose a different, more explicit route. We use the profile
  decomposition for $\psi^-_n(0)$ and then use \eqref{comp1} to
  produce a geometric profile decomposition for $\psi_n^+(0)$.  We
  apply Proposition \ref{lprof} to $\psi_n^-(0)$ to obtain the linear
  profile decomposition
  \[
  \psi_{n}^-(0)= \sum_{j=1}^l g^{n,j} \phi_j^- + w_{n,l}^-
  \]
  as described in the Proposition \ref{lprof}. This can be further
  factorized to $g^{n,j}=h^{n,j} e^{i t_{n,j} H^-}$ with $t_{n,j}
  \in \R$ and $h^{n,j} \in G$. Next, a standard diagonalisation
  argument, allows us to reduce the problem to the case when for each
  $j$, the sequence $\{t_{n,j}\}_{n \in \N}$ converges to some limit
  in $[-\infty,+\infty]$.  If for some $j$, $\lim_{n \rightarrow
    \infty} t_{n,j}=t_j$ is finite then one can shift $\phi_j$ by the
  propagator $e^{it_j H^-}$ and reduce the problem to the case
  $t_j=0$. Moreover, by absorbing the $e^{it_{n,j}H^-} \phi_j^- -
  \phi_j^-$ into the error $w_{n,l}^-$ one can further assume that
  $t_{n,j}=0, \forall n \in \N$.

  Next, it follows from \eqref{massdec} that
  \begin{equation} \label{summass} \sum_{j \geq 1} M(\phi_j^-) \leq
    \lim_{n \rightarrow \infty} M(\psi_n^-(0)) = m_0
  \end{equation}
  Therefore $\sup_{j} M(\phi_j^-) \leq m_0$. Assume first that, for
  some $\epsilon > 0$,
  \begin{equation} \label{massbound} \sup_{j} \| M(\phi_j^-) \| \leq
    m_0 - \epsilon
  \end{equation}
  We will show that this leads to a contradiction. Since $A(m)$ is
  monotone increasing and finite on $[0,m_0-\epsilon]$, by iv) in
  Proposition \ref{CT} it follows that
  \begin{equation} \label{Abound} A(m) \leq Bm, \quad 0 \leq m \leq
    m_0 - \epsilon
  \end{equation}
  where $B$ depends on $m_0$ and $\epsilon$ only.

  Then we introduce the nonlinear profiles $v_j^-: \R \times \R
  \rightarrow \C$ associated to $\phi_j^-$ and $t_j = \lim_{n
    \rightarrow \infty} t_{n,j}$ as follows

  - if $t_j=0$ which implies that $t_{n,j}=0$, then we define
  $\phi_j^+$ the compatible function in the sense of \eqref{compnew};
  the construction is done in part a) in Proposition \ref{recprop} and
  Lemma \ref{a2p2c}.  Then we define $v_j^\pm$ to be the
  maximal-lifespan solution of \eqref{psieq} with initial data
  $v_j^\pm(0)=\phi_j^\pm$.

  - if $t_{j}=+\infty$, we construct $\phi^+$ to be the compatible
  function in the sense of \eqref{asscat}.  Then we define $v_j^\pm$
  to be the the maximal-lifespan solution of \eqref{psieq} which
  scatters forward in time to $e^{itH^\pm} \phi_j^\pm$; the existence
  of this solution is guaranteed by part ii) of Theorem \ref{CT}.

  - if $t_{j}=-\infty$, the construction is similar to the one above.

  Using part a) in Proposition \ref{recprop} and Lemma \ref{a2p2c}, we
  construct $w_{n,l}^+$, the compatible pair to $w_{n,l}^-$ in the
  sense of \eqref{compnew}. Then we consider $w_{n,l}^\pm(t)$ the
  solution to \eqref{psieq} with $w_{n,l}^\pm$ initial data. By taking
  $l$ and $n$ large enough, it follows from \eqref{strvanish} and part
  vi) of Theorem \ref{CT} (the approximating solution is simply the
  free flow) that $w_{n,l}^\pm(t)$ is global in time.

  From \eqref{Abound} and Proposition \ref{CT} it follows that, for
  each $j$, $v_j^\pm$ are globally defined in time and satisfy
  \begin{equation} \label{vjinfo} M(v_j^\pm) = M(\phi_j^\pm) \leq m_0
    - \epsilon, \quad S(v_j^\pm) \leq A(M(\phi_j^\pm)) \leq B
    M(\phi_j^\pm)
  \end{equation} 
  Based on these nonlinear profiles we define the approximate global
  solutions
  \[
  \begin{split}
    \psi_{n,l}^\pm(t) = \sum_{j=1}^l T_{h^{n,j}} v_j^\pm(t+t_{n,j}) +
    w^\pm_{n,l}(t).
  \end{split}
  \]
  Our first claim is the following
  \begin{equation} \label{sumstrvj} \lim_{l \rightarrow \infty
    }\lim_{n \rightarrow \infty} S(\psi_{n,l}^\pm) = \lim_{l
      \rightarrow \infty } \sum_{j=1}^l S(v_j^\pm) \leq Bm_0
  \end{equation} 
  Indeed, by using \eqref{assstr} and the asymptotic orthogonality of
  $g^{n,j}$ we obtain that
  \[
  \lim_{n \rightarrow \infty} \| \sum_{j=1}^l T_{h^{n,j}}
  v_j^\pm(\cdot +t_{n,j}) \|_{L^4}^4 = \sum_{j=1}^l \| v_j^\pm
  \|_{L^4}^4
  \]
  Then using \eqref{strvanish}, \eqref{summass}, \eqref{massbound} and
  \eqref{Abound} we obtain \eqref{sumstrvj}.

  Our next claim is that
  \begin{equation} \label{appmas+} \lim_{n \rightarrow \infty}
    M(\psi_{n,l}^\pm(0) - \psi_n^\pm(0))=0, \qquad \forall l \in \N.
  \end{equation}
  We have from the construction that
  \[
  \psi_{n,l}^-(0) = \psi_n^-(0)
  \]
  Therefore we need to show only the $+$ part in \eqref{appmas+}. We
  prove this by using the variable $\psi_2$, hence we involve the
  compatibility conditions \eqref{comp1}. For this we define for each
  $j \in \{1,..,l\}$
  \[
  \begin{split}
    \frac{\psi_{2,n,j}(t)}{r}:= & \frac{T_{h^{n,j}} v_j^+(t+t_{n,j}) - T_{h^{n,j}} v_j^-(t+t_{n,j})}{2i} \\
    A_{2,n,j}(t,r):= & -1 + \int_0^r \frac{|T_{h^{n,j}}
      v_j^+(t+t_{n,j})|^2 - |T_{h^{n,j}} v_j^-(t+t_{n,j})|^2}4 sds
  \end{split}
  \]
  and
  \[
  \begin{split}
    \frac{\psi_{2,er}(t)}{r}: &= \frac{w^+_{n,l}(t)- w^-_{n,l}(t)}{2i} \\
    A_{2,er}(t,r): &= - 1 + \int_0^r \frac{|w^+_{n,l}(t)|^2 -
      |w^-_{n,l}(t)|^2}4 sds
  \end{split}
  \]
  At time $t=0$ it is more convenient to define
  \[
  \begin{split}
    \psi_{n,l,2}= \sum_{j=1}^l \psi_{2,n,j} + \psi_{2,er}, \quad
    1+\tilde A_{n,l,2} = \sum_{j=1}^l (1+A_{2,n,j}) + (1+A_{2,er})
  \end{split}
  \]
  Indeed, from \eqref{A2}, the definition of $A_{n,l,2}$ should have
  been different, but the above one works better for us for technical
  purposes. We make the following claims (at $t=0$):
  \begin{equation} \label{claim1} \lim_{n \rightarrow \infty} \|
    \psi_{2,n,j_1} \psi_{2,n,j_2} \|_{L^\infty} + \| \psi_{2,n,j_1}
    \psi_{2,er} \|_{L^\infty} =0 , \quad \forall j_1 \ne j_2
  \end{equation}
  \begin{equation} \label{claim2} \lim_{n \rightarrow \infty} \|
    \frac{\psi_{2,n,j_1} \psi_{2,n,j_2}}r \|_{L^2} + \|
    \frac{\psi_{2,n,j_1} \psi_{2,er}}r \|_{L^2} =0, \quad \forall j_1
    \ne j_2
  \end{equation}
  \begin{equation} \label{claim3} \lim_{n \rightarrow \infty} \|
    \frac{\psi_{2,n,j_1} \psi_{2,n,j_2}}{r^2} \|_{L^1} + \|
    \frac{\psi_{2,n,j_1} \psi_{2,er}}{r^2} \|_{L^1} =0, \quad \forall
    j_1 \ne j_2
  \end{equation}
  and the similar ones involving $1+A_{2,n,j}, 1+A_{2,er}$ instead of
  $\psi_{2,n,j}, \psi_{2,er}$.

  We will establish \eqref{claim1}, the other two being similar. We
  will need the following estimate, whose proof is provided in the
  Appendix. If $g^n=g_{\l_n,t_n}, n \in \N$ is asymptotically
  orthogonal to the constant sequence $g_{1,0}$ then for every $f \in
  L^2$ and any compact set $K \subset (0,+\infty)$ the following holds
  true for
  \begin{equation} \label{ps} \lim_{n \rightarrow \infty} \| g^n f
    \|_{L^2(K)} = 0.
  \end{equation}
  Let $j_1 \ne j_2$. \eqref{claim1} is scale invariant, hence by
  rescaling in space and by shifting time we can reduce the problem to
  the case $\l_{n,j_1}=1$ and $t_{n,j_1}=0$ for all $n \in \N$.  We
  obviously have from \eqref{2est} that
  \[
  \| \psi_{2,n,j_1} \|_{L^\infty} + \| \psi_{2,n,j_2} \|_{L^\infty}
  \les \| \phi^{j_1} \|_{L^2} + \| \phi^{j_2} \|_{L^2} \les 1
  \]
  We need to localize this estimate.  Given $\epsilon > 0$ choose $R$
  such that
  \[
  \| \phi^{j_1} \|_{L^2(\R_+ \setminus [R^{-1},R])} \leq \epsilon
  \]
  Using \eqref{loc2est} we obtain that
  \[
  \| \psi_{2,n,j_1} \|_{L^\infty(\R_+ \setminus [\epsilon R^{-1},R])}
  \les \epsilon.
  \]
  From \eqref{ps} it follows that there exists $N(\epsilon)$ such that
  if $n \geq N(\epsilon)$ the following holds true
  \[
  \| g^{n,j_2} \phi^{j_2} \|_{L^2[\epsilon R^{-1},\epsilon^{-1}R]}
  \leq \epsilon
  \]
  and hence by invoking \eqref{loc2est} it follows that
  \[
  \| \psi_{2,n,j_2} \|_{L^\infty([\epsilon R^{-1},R])} \les \epsilon
  \]
  From the above estimates we obtain
  \[
  \begin{split}
    & \| \psi_{2,n,j_1} \psi_{2,n,j_2} \|_{L^\infty}  +  \| \psi_{2,n,j_1} \psi_{2,er} \|_{L^\infty} \\
    \les & \| \psi_{2,n,j_1}\|_{L^\infty(\R_+ \setminus [\epsilon
      R^{-1},R])} \| \psi_{2,n,j_2} \|_{L^\infty}
    +  \| \psi_{2,n,j_1}\|_{_{L^\infty}} \psi_{2,er} \|_{L^\infty([\epsilon R^{-1},R])} \\
    \les & \epsilon
  \end{split}
  \]
  which is true provided that $n \geq N(\epsilon)$ (which was defined
  earlier). This implies the first half of \eqref{claim1}.

  The second part of \eqref{claim1}, which contains the terms
  involving $\psi_{2,er}$, follows by involving similar arguments
  together with \eqref{weak}.

  Now we return to the proof of \eqref{appmas+}. $\psi_{2,n,l}(0)$
  obeys the following differential equation:
  \[
  \begin{split}
    \partial_r \psi_{n,l,2} & = i \tilde A_{n,l,2} \psi_{n,l}^- - \frac{1}r  \tilde A_{n,l,2} \psi_{n,l,2} + E_{n,l}^1 \\
    \partial \tilde A_{n,l,2} & = \Im{(\psi^-_{n,l}
      \overline{\psi_{n,l,2}} )} + \frac1r (1-\tilde A_{n,l,2}^2) +
    E_{n,l}^2
  \end{split}
  \]
  where, based on \eqref{claim1}-\eqref{claim3},
  \[
  \lim_{n \rightarrow \infty} \| |E_{n,l}^1| +|E_{n,l}^2| \|_{L^2(rdr)
    \cap L^1(dr)} =0
  \]
  Based on these, we can now use \eqref{2estdif} for $\psi_{n,2}$ and
  $\psi_{n,l,2}$ (recall that $\psi_{n,l}^-(0)=\psi_{n}^-(0)$), and
  obtain that
  \[
  \lim_{n \rightarrow \infty} \| \frac{\psi_{n,2}(0)-\psi_{n,l,2}(0)}r
  \|_{L^2} =0
  \]
  In turn this implies \eqref{appmas+}.

  Next we turn our attention to the dynamical properties of our
  approximate solutions.  In this case, using the convention made on
  the definition of $N^\pm (\psi^\pm_{n,l})$, we remark that
  \[
  A_{n,l,2} = -1 + \int_0^r \frac{|\psi_{n,l}^+|^2 -
    |\psi_{n,l}^-|^2}4 sds
  \]
  with is different than the $\tilde A_{n,l,2}$ (at time $t=0$) we
  have used above.

  We claim that $\psi_{n,l}^\pm$ approximately solve \eqref{psieq} in
  the following sense
  \begin{equation} \label{appstr} \lim_{l \rightarrow \infty} \lim
    \sup_{n \rightarrow \infty} \| (i\partial_t + H^\pm)\psi_{n,l}^\pm
    - N^{\pm}(\psi_{n,l}^\pm) \|_{L^\frac43}=0
  \end{equation} 
  Using the equations that the components of $\psi_{n,l}^\pm$ obey,
  this can be broken into two statements:
  \[
  \lim_{l \rightarrow \infty} \lim \sup_{n \rightarrow \infty} \|
  \sum_{j=1}^l N^\pm( T_{h^{n,j}} v_j^\pm(t+t_{n,j})) -
  N^{\pm}(\sum_{j=1}^l T_{h^{n,j}} v_j^\pm(t+t_{n,j}))
  \|_{L^\frac43}=0
  \]
  and
  \[
  \lim_{l \rightarrow \infty} \lim \sup_{n \rightarrow \infty} \|
  N^\pm( \psi_{n,l}^\pm - w^\pm_{n,l}) - N^{\pm}(\psi^\pm_{n,l})
  \|_{L^\frac43}=0
  \]
  The second statement follows from \eqref{Ndif}, \eqref{strvanish}
  and \eqref{sumstrvj}.

  The first statement follows from \eqref{assort}, \eqref{assstr},
  \eqref{assrdrm} and \eqref{assint} by identifying all terms
  involved:

  - cubic cross terms of type $T_{h^{n,j_1}} v_{j_1}^\pm(t+t_{n,j_1})
  T_{h^{n,j_2}} v_{j_2}^\pm(t+t_{n,j_2}) T_{h^{n,j_3}}
  v_{j_3}^\pm(t+t_{n,j_3})$ with at least two of the $j_1,j_2,j_3$
  distinct coming either from the local part of $A_0$ ($- \frac12
  \Re(\overline{\psi}^+ \psi^-) \psi^\pm$) or from or from the
  $\psi_2$ part ($\frac1{r}\Im{(\psi_2 \bar{\psi}^\pm)} \psi^\pm$);
  these are treated with \eqref{assstr}.

  - nonlocal cubic terms coming from the nonlocal part of $A_0$ of
  type

  $T_{h^{n,j_1}} v_{j_1}^\pm(t+t_{n,j_1}) [r\partial_r]^{-1} \left(
    T_{h^{n,j_2}} v_{j_2}^\pm(t+t_{n,j_2}) T_{h^{n,j_3}}
    v_{j_3}^\pm(t+t_{n,j_3}) \right)$; they are treated using
  \eqref{assstr} and \eqref{assrdrm}.

  - nonlocal cubic terms coming from $1+A_2$; they are of type
  \[
  T_{h^{n,j_1}} v_{j_1}^\pm(t+t_{n,j_1}) r^{-2} [r^{-1}
  \bar \partial_r]^{-1} \left( T_{h^{n,j_2}} v_{j_2}^\pm(t+t_{n,j_2})
    T_{h^{n,j_3}} v_{j_3}^\pm(t+t_{n,j_3}) \right)
  \]
  and are treated using \eqref{assstr} and \eqref{assint}.

  Because of \eqref{appmas+} and \eqref{appstr} we can invoke the
  stability result in part vi) of Proposition \ref{CT}, and obtain a
  contradiction with the fact that $\lim_{n \rightarrow \infty}
  S(\psi_n^\pm)=\infty$. Therefore the assumption \eqref{massbound} is
  false. In light of \eqref{summass}, it follows that the only
  possibilities are

  i) $\phi^j=0, \forall j$; this is impossible due to global in time
  well-posedness of the solution $w_{n}^\pm$ (earlier denoted by
  $w_{n,l}^\pm$).

  ii) after a relabeling, $\phi^1=\phi, M(\phi)=m_0, \phi^j=0, \forall
  j \geq 2$.  In this case the linear profile decomposition simplifies
  to
  \[
  \psi_n^-(0)= h^n e^{it_n H^-} \phi + w_n
  \]
  with $\lim_{n \rightarrow \infty} M(w_n)=0$ (which obviously implies
  \eqref{strvanish}).  If $\lim_{n \rightarrow \infty} t_n =0$ then we
  are done. We are left with the case when $\lim_{n \rightarrow
    \infty} t_n =\infty$, as the case with $-\infty$ is entirely
  similar. An easy argument shows that
  \[
  \lim_{n \rightarrow \infty} S_{\geq 0} (e^{itH^-} h^n e^{it_n H^-}
  \phi) = 0
  \]
  and since $\lim_{n \rightarrow \infty} S (e^{itH^-} w_n)=0$, we can
  invoke the stability argument with $0$ as the approximate solution
  and $\psi_n^\pm(0)$ the initial data to derive that for $n$ large
  enough $S(e^{itH^\pm} \psi_n^\pm(0))$ is finite which contradicts
  the hypothesis.
\end{proof}

We end the section with the proof of Theorem \ref{lambdacontrol}. Here
we adapt the corresponding argument in \cite{KeMe1}.

\begin{proof}[Proof of Theorem \ref{lambdacontrol}]
  i) Consider a global forward in time critical element $\psi^\pm_C$
  which satisfies Theorem \ref{cc}, i.e.
  \[
  \psi^\pm_{comp}(t)= \frac{1}{\lambda(t)}
  \psi^\pm_C(\frac{r}{\lambda(t)},t) \in K
  \]
  where $K$ is compact.

  Assume that there is a sequence of times $t_n \rightarrow \infty$
  such that $\lim_{n \rightarrow \infty} \lambda(t_n)=0$. Given the
  statement on Theorem \ref{cc}, we can select a subsequence (which we
  label the same way) such that $\lim_{n \rightarrow \infty} t_n
  =+\infty$ and
  \[
  \lambda(t_n) \leq 2 \inf_{t \in [0,t_n]} \lambda(t)
  \]
  Using again Theorem \ref{cc}, we obtain (after passing again to a
  subsequence labeled the same way) that
  \[
  \psi^\pm_{0,n}:= \frac{1}{\lambda(t_n)}
  \psi^\pm_C(\frac{r}{\lambda(t_n)},t_n) \rightarrow \psi^\pm_0
  \]
  in $L^2$. Using the mass conservation it follows that
  $M(\psi_0^\pm)=m_0 > 0$.  The fact that $\psi^\pm_0$ satisfy the
  compatibility relations can be done in an indirect way. We redefine
  $\psi^+_0$ to be the compatible pair of $\psi^-_0$, and by using the
  last statement in Lemma \ref{a2p2c}, see \eqref{2estdif}, it follows
  that $\psi^+_{0,n} \rightarrow \psi^+_0$ in $L^2$.

  Let $\psi^\pm_{n} ,\psi^\pm$ be the solution of \eqref{psieq} with
  initial data $\psi^\pm_{0,n},\psi^\pm_0$ at $t=0$. From part v) of
  Proposition \ref{CT} it follows that $T_{\pm}(\psi^\pm_{0}) \geq
  \underline{\lim}_{n \rightarrow \infty} T_{\pm}(\psi^\pm_{0,n}) $.
  Since the former are $\infty$, it follows that
  $T_{\pm}(\psi^\pm_{0})=\infty$, hence $\psi^\pm$ is global. By the
  Cauchy theory $\| \psi^\pm \|_{S(\R \times \R)}=+\infty$, since
  otherwise the stability part of the Cauchy theory would imply that
  $\| \psi^\pm_C \|_{S(\R \times \R)} = \| \psi_n \|_{S(\R \times \R)}
  < + \infty$, a contradiction.

  From the Cauchy theory we also have that for every $t \in \R$,
  \begin{equation} \label{limit} \lim_{n \rightarrow \infty}
    \psi_n^\pm(t)=\psi^\pm(t)
  \end{equation}

  By uniqueness of the Cauchy problem \eqref{psieq} it follows that
  \[
  \psi^\pm_n (r,t)= \frac{1}{\lambda(t_n)}
  \psi^\pm_C(\frac{r}{\lambda(t_n)},t_n + \frac{t}{\lambda^2(t_n)})
  \]

  We claim that $\lim_{n \rightarrow \infty} t_n \lambda^2(t_n) =
  \infty$. If this were not true, then, after passing to a
  subsequence, we would have $\lim_{n \rightarrow \infty} t_n
  \lambda^2(t_n) = t_0$.  Then, by \eqref{limit},
  \[
  \psi^\pm_n(\cdot,-t_n \lambda^2(t_n)) = \frac{1}{\lambda(t_n)}
  \psi^\pm_C(\frac{\cdot}{\lambda(t_n)}, 0 ) \rightarrow
  \psi^\pm(\cdot, -t_0)
  \]
  in $L^2$. Since $\lim_{n \rightarrow \infty} \lambda(t_n)=0$, this
  implies $\psi^\pm(t_0) \equiv 0$ which is in contradiction with
  $M(\psi^\pm(-t_0))=M(\psi^\pm_0)=m_0 > 0$.

  Next, fix $t \in (-\infty,0]$ and note that for $n$ large enough
  $t_n + \frac{t}{\lambda^2(t_n)} > 0$, hence
  \[
  \tilde \lambda_n(t)=\frac{\lambda(t_n +
    \frac{t}{\lambda^2(t_n)})}{\lambda(t_n)} \geq \frac12
  \]
  We claim that $\lim \sup_{n \rightarrow \infty} \tilde \lambda_n(t)
  < + \infty$. Assume this is not the case; then there is an
  increasing subsequence, still denoted by $\tilde \lambda_n(t)$, such
  that $\lim_{n \rightarrow \infty} \tilde \lambda_n(t)=+\infty$. On
  the other hand we have
  \[
  \begin{split}
    \frac1{\tilde \lambda_n (t)} \psi_n^\pm(\frac{r}{\tilde
      \lambda_n(t)},t)
    & = \frac1{\lambda(t_n + \frac{t}{\lambda^2(t_n)})} \psi^\pm_C(\frac{r}{\lambda(t_n + \frac{t}{\lambda^2(t_n)})}, t_n + \frac{t}{\lambda^2(t_n)}) \\
    & = \psi^\pm_{comp}(r, t_n + \frac{t}{\lambda^2(t_n)}) \in K
  \end{split}
  \]
  Since $\| \frac1{\tilde \lambda(t_n)} \psi_n^\pm(\frac{r}{\tilde
    \lambda(t_n)},t) - \frac1{\tilde \lambda(t_n)}
  \psi_0^\pm(\frac{r}{\tilde \lambda(t_n)},t)\|_{L^2} = \|
  \psi^\pm_n(t) - \psi^\pm(t) \|_{L^2}$ and $\frac1{\tilde
    \lambda(t_n)} \psi_0^\pm(\frac{r}{\tilde \lambda(t_n)},t)
  \rightarrow 0$ in $L^2$ (here we use that $\tilde \lambda_n(t)
  \rightarrow \infty$), it follows that $\psi^\pm_{comp}(r, t_n +
  \frac{t}{\lambda^2(t_n)}) \rightarrow 0$ in $L^2$ which violates the
  fact fact that $M(\psi^\pm_{comp}(t')) = m_0$ for all $t' \in \R$.

  Therefore we have established that $\frac12 \leq \tilde \lambda_n(t)
  < + \infty$ for each $t < 0$.  On a subsequence we have that $\tilde
  \lambda_n(t) \rightarrow \tilde \lambda(t)$ and by \eqref{limit}
  \[
  \frac1{\tilde \lambda (t)} \psi^\pm(\frac{r}{\tilde \lambda(t)},t)
  \in K
  \]
  with $\tilde \lambda(t) \geq \frac12$ which creates the solution
  claimed by part i) after applying the time reversal transformation
  $t \rightarrow -t$.

  ii) It suffices to consider the case $T_+ < \infty$ in which case
  the forward maximal life-span of the solution is $I_+=[0,T_+)$.  If
  $\lim_{t \rightarrow T_+} \lambda(t) \ne \infty$, then there is an
  increasing sequence $t_n \rightarrow T_+$ such that $\lambda(t_n)
  \rightarrow \lambda_0 \in [0,+\infty)$.

  By the compactness of $\bar K$, it follows that
  $\psi^\pm_n=\frac{1}{\lambda(t_n)}
  \psi^\pm_C(\frac{r}{\lambda(t_n)}, t_n) \rightarrow \psi^\pm_0$.
  Let $\psi^\pm,\psi^\pm_n$ be the solutions of \eqref{psieq} with
  initial data $\psi^\pm_0,\psi^\pm_{0,n}$ at time $t=T_+$.
  $\psi^\pm$ is defined on some $[T_+ - \delta, T_+ + \delta]$ with
  $\delta > 0$ and, by the Cauchy theory, $\psi^\pm_n$, for large
  enough $n$, is defined on $[T_+ - \frac{\delta}2, T_+ +
  \frac{\delta}2]$ with $\| \psi_n \|_{S([T_+ - \frac{\delta}2, T_+ +
    \frac{\delta}2])} < \infty$. On the other hand,
  \[
  \psi^\pm_C(r,t)= \lambda(t_n) \psi^\pm_n(\lambda(t_n)r, T_+ + t-t_n)
  \]
  and by choosing $n$ large enough so that $T_+ \leq t_n +
  \frac{\delta}2$ it follows that $\| \psi_C^\pm \|_{S[t_n,T_+)} <
  \infty$ which is a contradiction with the assumption that
  $\psi^\pm_C$ blows-up at time $T_+$.

\end{proof}
\section{Momentum and localized momentum.}
In this section we rule out the possible scenarios exhibited in
Theorem \ref{lambdacontrol}.  With the language used in Section
\ref{concomp}, we claim the following
\begin{theo} \label{lack} If $m_0 < 4$ critical elements do not exist.
\end{theo}

This will be based on virial type identities. Virial identities for
the Schr\"odinger Map problem originate in the work of Grillakis and
Stefanopoulos via a Lagrangian approach, see \cite{GM}. In their work
the formulation of these identities is at the level of the conformal
coordinate, obtained by using the stereographic projection. Our
approach is different in the sense that we derive the virial
identities at the level of the gauge components. However our results
can be derived from \cite{GM}.

\subsection{Virial type identities.}
This section is concerned with identities involving solutions of
\eqref{psieq} which satisfy the compatibility condition
\eqref{compnew}.

Given $a: \R_+ \rightarrow R$ a smooth function, i.e.
$|(r\partial_r)^\alpha| \les_\alpha 1$, and which decays at infinity
we claim that
\begin{equation} \label{vir1} \frac{d}{dt} \int a(r) (1+A_2) r dr = -
  \int r \partial_r a(r) \Re( \psi_1 \frac{\bar{\psi}_2}r) rdr
\end{equation}

By using part i) of Theorem \ref{CT-CC}, the proof of \eqref{vir1}
goes as follows
\[
\begin{split}
  \frac{d}{dt} \int a(r) (1+A_2) r dr & = \int a(r) \partial_t A_2 rdr
  = \int a(r) \Im{(\psi_0 \bar{\psi}_2)}  r dr \\
  & = \int a(r)  \Im(i(\partial_r \psi_1 + \frac{1}r \psi_1 + \frac{iA_2}{r^2} \psi_2) \bar{\psi}_2) rdr \\
  & = \int a(r) \Im(i ( r \partial_r \psi_1 + \psi_1) \bar{\psi}_2) dr \\
  & = \int a(r)  \left( \Im(i \partial_r ( r \psi_1 \bar{\psi}_2)) - \Im(i r \psi_1 \partial_r \bar{\psi}_2) \right) dr \\
  & = - \int \partial_r a(r) \Im(i \psi_1 \bar{\psi}_2) rdr = - \int
  r \partial_r a(r) \Re( \psi_1 \frac{\bar{\psi}_2}r) rdr
\end{split}
\]
This computation is valid in a classical sense provided that $R_\pm
\psi^\pm \in H^2$.  For general functions $\psi^\pm$ this is done by
using a regularization argument as we did in the proof of part i) of
Theorem \ref{CT}. Note that the quantities involved on both sides of
\eqref{vir1} are meaningful in light of the fact that $\psi_0 \in \dot
H^{-1}_e$ and $a \psi_2 \in \dHe$.

We now introduce the two momenta, the radial and the temporal one, as
follows
\[
M_1 = \frac{\Re{(\psi_1 \bar{\psi}_2)}}{1-A_2} , \qquad M_0= -
\frac{\Re{(\psi_0 \bar{\psi}_2)}}{1-A_2}
\]
Using the covariant calculus, the time momentum can be further written
as follows
\[
\begin{split}
  -(1-A_2) M_0 &  =  \Re{(\psi_0 \bar{\psi}_2)} \\
  & = \Re{\left( i(D_1 \psi_1 + \frac{1}r \psi_1 + \frac{1}{r^2} D_2 \psi_2) \bar{\psi}_2 \right)} \\
  & = - \Im{(\partial_r  \psi_1 \bar{\psi}_2)} - \frac1r \Im{(\psi_1 \bar{\psi}_2)} - \frac{A_2}{r^2} |\psi_2|^2 \\
  & = - \partial_r \Im(\psi_1 \bar{\psi}_2) - \Im(\psi_1 \partial_r \bar{\psi}_2)  - \frac1r  \partial_r A_2 - \frac{A_2}{r^2} |\psi_2|^2 \\
  &= - \partial_r^2 A_2 - \frac1r \partial_r A_2 - A_2(|\psi_1|^2+
  \frac{|\psi_2|^2}{r^2})
\end{split}
\]
which leads to
\begin{equation}
  M_0= - \Delta \ln(1-A_2) - \left(\frac{ \partial_r A_2}{1-A_2} \right)^2 + \frac{A_2}{1-A_2} (|\psi_1|^2+ \frac{|\psi_2|^2}{r^2}) 
\end{equation}

The following identity plays a fundamental role in our analysis
\begin{equation} \label{vir2}
  \partial_t M_1 - \partial_r M_0 = - \partial_r A_0
\end{equation}
This is established by using the covariant rules of calculus,
\[
\begin{split}
  \partial_t M_1 & = \frac{\Re{(D_0 \psi_1 \bar{ \psi}_2)}}{1-A_2} +
  \frac{\Re{(\psi_1 \overline{ D_0 \psi_2})}}{1-A_2}
  + \frac{\Re{(\psi_1 \bar{\psi}_2)}}{(1-A_2)^2} \partial_t A_2 \\
  & = \frac{\Re{(D_1 \psi_0 \bar{ \psi}_2)}}{1-A_2} +
  \frac{\Re{(\psi_1 \overline{ D_2 \psi_0})}}{1-A_2}
  + \frac{\Re{(\psi_1 \bar{\psi}_2)}}{(1-A_2)^2} \Im{(\psi_0 \bar{\psi}_2)} \\
  & = \partial_r M_0 - \frac{\Re{(\psi_0 \partial_r \bar{
        \psi}_2)}}{1-A_2} - \frac{\Re{(\psi_0 \bar{
        \psi}_2)}}{(1-A_2)^2} \partial_r A_2 + \frac{\Re{(\psi_1
      \overline{ D_2 \psi_0})}}{1-A_2}
  + \frac{\Re{(\psi_1 \bar{\psi}_2)}}{(1-A_2)^2} \Im{(\psi_0 \bar{\psi}_2)} \\
  & = \partial_r M_0 - \frac{A_2 \Im{(\psi_0 \bar{ \psi}_1)}}{1-A_2} -
  \frac{\Re{(\psi_0 \bar{ \psi}_2)}}{(1-A_2)^2} \Im{(\psi_1
    \bar{\psi}_2)} + \frac{A_2 \Im{(\psi_1 \overline{\psi_0})}}{1-A_2}
  + \frac{\Re{(\psi_1 \bar{\psi}_2)}}{(1-A_2)^2} \Im{(\psi_0 \bar{\psi}_2)} \\
  & = \partial_r M_0 - 2 \frac{A_2 \Im{(\psi_0 \bar{ \psi}_1)}}{1-A_2}
  + \frac{|\psi_2|^2\Im{(\psi_0 \bar{\psi}_1)}}{(1-A_2)^2}  \\
  & = \partial_r M_0 + \Im{(\psi_0 \bar{ \psi}_1)} \\
  & = \partial_r M_0 -  \partial_r A_0 \\
\end{split}
\]
The above computation is meaningful provided that $R_\pm \psi^\pm \in
H^3$.

Next we derive a localized version of \eqref{vir2} which has also the
advantage that it makes sense for $\psi^\pm \in L^2$ only. We take $a:
\R_+ \rightarrow \R$ to be a smooth function which decays at infinity
and satisfies also $|\frac1r \partial_r a | \les 1$ and
$|\partial_r^2| \les 1$. As a consequence we have that if $f \in \dHe$
then $\frac1r f \partial_r a \in \dHe$.

We multiply \eqref{vir2} by $a$ and integrate by parts as follows
\begin{equation} \label{vir3} \int a(r) M_1(r) dr \left|_0^T \right. +
  \int_0^T \int \partial_r a(r) M_0 dr = \int_0^T \int \partial_r a(r)
  A_0 dr
\end{equation}
This identity is now meaningful for $\psi^\pm \in L^2$. Indeed each
term is well-defined for the following reasons:

- the first since $a$ is bounded and $\frac1r M_1 \in L^2$,

- the second since $\psi_0 \in \dot H_e^{-1}$ and $\frac1r \partial_r
a \cdot \psi_2 \in \dHe$,

- the third since $\frac1r \partial_r a$ is bounded and $A_0 \in L^1$.

The justification of \eqref{vir3} for general $\psi^\pm \in L^2$ is
done by regularizing $\psi^\pm$ as above.

It will be useful to rewrite the second term on the left-hand side as
follows
\[
\begin{split}
  \int \partial_r a(r) M_0 dr & = \int \frac{1}r \partial_r a(r) \left( - \Delta \ln(1-A_2) - \left(\frac{ \partial_r A_2}{1-A_2} \right)^2 + \frac{A_2}{1-A_2} (|\psi_1|^2+ \frac{|\psi_2|^2}{r^2}) \right) rdr \\
  & =  \int \partial_r (\frac{1}r \partial_r a(r)) \partial_r \ln(1-A_2) rdr \\
  & - \int \frac{1}r \partial_r a(r) \left( \left(\frac{ \partial_r
        A_2}{1-A_2} \right)^2 - \frac{A_2}{1-A_2} (|\psi_1|^2+
    \frac{|\psi_2|^2}{r^2}) \right) rdr
\end{split}
\]

\subsection{Proof of Theorem \ref{lack}.} The argument is in the
spirit of the corresponding one in \cite{KeMe1}.

Based on a localized version of \eqref{vir1} and \eqref{vir3} we rule
out the possibilities exhibited in parts i) and ii) of Theorem
\ref{lambdacontrol}.

In order to do so, we make a few important remarks. Since we are in
the setup of $E(u) < 4 \pi$, which translates into $\| \psi^\pm
\|_{L^2} < 2$, we obtain from \eqref{A2neg} that
\[
\frac{-A_2}{1-A_2} \ges 1
\]
Next, by using \eqref{rela} and \eqref{loc2est}, the concentration
property \eqref{psicon} implies
\[
\int_{r \ges C(\eta) c^{-1}} (| \psi_1(r) |^2 +
\frac{|\psi_2(r)|}{r^2} + \frac{(A_2(r)+1)^2}{r^2}) rdr \les \eta,
\qquad \forall t \in I_+.
\]
We start by ruling out the existence of a critical element from part
i) of Theorem \ref{lambdacontrol}, i.e. the global element with
$\lambda(t) \geq c > 0, \forall t > 0$.  In \eqref{vir3}, we take
$a(r)=r^2 \phi(\frac{r}{R})$ where $\phi$ is smooth and equals $1$ for
$ r \leq 1$ and $0$ for $r \leq 2$, and obtain
\begin{equation} \label{virmid}
  \begin{split}
    \int a(r) M_1(r) dr \left|_0^T \right. & = \int_0^T \int \partial_r  (\frac{1}r \partial_r a(r)) \partial_r \ln(1-A_2) rdr dt  \\
    & + \int_0^T \int  \frac{1}r \partial_r a(r) \left( \left(\frac{ \partial_r A_2}{1-A_2} \right)^2 - \frac{A_2}{1-A_2} (|\psi_1|^2+ \frac{|\psi_2|^2}{r^2}) \right) rdr  dt \\
    & + \int_0^T \int \partial_r a(r) A_0 dr dt
  \end{split}
\end{equation}
In this identity there are two main terms which we compare against
each other: the first on the left-hand side and the second on the
right-hand side. All the other terms are controlled by one of the two
main terms just mentioned.

We choose $\eta \ll 1$ small enough (the exact choice is derived from
the inequalities on the error terms below) and $R = C(\eta) c^{-1} \gg
c^{-1}$; we estimate the main terms in the above expression by
\[
|\int a(r) M_1 dr | \les \int r^2 |\psi_1| |\frac{\psi_2}r| rdr \les
R^2 \| \psi_1 \|_{L^2} \| \frac{\psi_2}r \|_{L^2} \les R^2 E
\]
which is valid both at $t=0$ and $t=T$, and
\[
\int_0^T \int \frac{1}r \partial_r a(r) \left( \left(\frac{ \partial_r
      A_2}{1-A_2} \right)^2 - \frac{A_2}{1-A_2} (|\psi_1|^2+
  \frac{|\psi_2|^2}{r^2}) \right) rdr dt \ges T E
\]
By choosing $T \gg R^2$ we obtain a contradiction, provided that we
establish that all the other terms involved in \eqref{virmid} are of
error type.

The first term on the left-hand side of \eqref{virmid} is bounded as
follows
\[
\begin{split}
  |\int_0^T \int \partial_r (\frac{1}r \partial_r a(r)) \partial_r
  \ln(1-A_2) rdr dt| \les \int_0^T \int_{r \approx R} |\partial_r A_2|
  dr dt \les T \eta^\frac12 \ll TE
\end{split}
\]
For the third term on the right-hand side of \eqref{virmid} we use
\eqref{A0c} and write
\[
|\int_0^T \int \partial_r a(r) A_0 dr dt| =|\int_0^T \int (-2+
\frac1r \partial_r a(r)) A_0 r dr dt|
\]
The coefficient $-2+
\frac1r \partial_r a(r)$ is supported in $r \geq R$, so the above 
expression is further bounded by
\[
\les \int_0^T \| \psi_1 \|_{L^2[R, \infty)} \| \frac{\psi_2}r
\|_{L^2[R, \infty)} dt \les T \eta \ll TE
\]
We have just shown that the other two terms in \eqref{virmid} are of
error type and this finishes the contradiction argument. Therefore we have shown 
that the scenario exhibited in part i) of Theorem \ref{lambdacontrol} cannot happen.

Next we rule out the critical element of type exhibited in part
ii). In this case the assumption is that we have a critical element
with $T_+ <\infty, \lim_{t \rightarrow T_+} \lambda(t)=+\infty$.

For fixed $R$ we claim that
\begin{equation} \label{limitT} \lim_{t \rightarrow T_+} \int
  \phi(\frac{r}R) (1+A_2) rdr = 0
\end{equation}
Indeed, for given $\epsilon > 0$, pick $\eta$ such that $\eta R^2 <
\epsilon$.  Recalling that $|u_1|^2+|u_2|^2=|\psi_2|^2$ and using
\eqref{loc2est} we obtain
\[
\begin{split}
  & \| \phi(\frac{r}R) (1+A_2) \|_{L^1} \\
  \les & (C(\eta) \lambda(t)^{-1})^2 \| \frac{1+A_2}{r}
  \|_{L^2(0,C(\eta)\lambda^{-1}(t)]}
  + R^2 \| \frac{1+A_2}{r} \|_{L^2[C(\eta) \lambda^{-1}(t),R]}  \\
  \les & (C(\eta) \lambda(t)^{-1})^2 E + \eta R^2
\end{split}
\]
By choosing $t$ close enough to $T_+$, we obtain $ (C(\eta)
\lambda(t)^{-1})^2 E < \epsilon$, and this establishes \eqref{limitT}.

Next we choose $a(r)=\phi(\frac{r}R)$, fix $\eta > 0$, integrate
\eqref{vir1} on $[t,T_+)$ and use \eqref{limitT} to obtain
\[
\int \phi(\frac{x}R) (1+A_2(r,t)) rdr \les (T_+ - t) \| \psi_1(t)
\|_{L^2(|x| \approx R)} \| \frac{\psi_2(t)}r \|_{L^2(|x| \approx R)}
\les (T_+ - t) \eta
\]
provided that $R \ges C(\eta)\lambda(t)^{-1}$. By fixing $t$ and
taking $\eta \rightarrow 0$ (which also forces $R \rightarrow
\infty$), it follows that
\[
\int (1+A_2(r,t)) rdr = 0
\]
which implies $A_2(t) \equiv -1$ hence, by \eqref{cons} and then by
\eqref{comp} it follows that $\psi_2(t) \equiv 0$ and $\psi_1(t)
\equiv 0$. Finally this implies by \eqref{rela} that $\psi^\pm(t)=0$
which contradicts the blow-up hypothesis at time $T_+$ (since the
solution is globally in time $\equiv 0$).

\section{Proof of the main result}

This section is dedicated to the proof of Theorem \ref{MT}. Given an
initial data $u_0 \in \dot H^1 \cap \dot H^3$, by using Theorem
\ref{clasic} it follows that it has a unique local solution on $[0,T]$
for some $T >0$.  On this interval we use sections \ref{CG} and
\ref{COG} to construct the associated fields $\psi^\pm$ obeying the
system \eqref{psieq} and whose mass satisfies $\| \psi^+_0 \|_{L^2}=\|
\psi^-_0 \|_{L^2} < 2$. By using Theorem \ref{lack} (and the previous
reduction from Section \ref{concomp}) it follows that the solution
$\psi^\pm$ is globally defined on $[0,+\infty)$ and with $\| \psi^\pm
\|_{L^4(\R_+ \times \R_+)} < + \infty$. By part vii) of Theorem
\ref{CT} the $H^2$ regularity of $R_\pm \psi^\pm_0$ is propagated at
all times $t \geq 0$. Invoking Proposition \ref{greg} this implies
that $u(t) \in \dot H^1 \cap \dot H^3$ with bounds depending on $\|
\psi^\pm \|_{L^4(\R_+ \times \R_+)} , \| R_\pm \psi^\pm_0 \|_{H^2}$
and $t$. Using again Theorem \ref{clasic}, this means that the
solution $u(t)$ can be continued past time $T$ and in fact for all
times $t \geq 0$ with $u(t) \in L^\infty_t (\R_+: \dot H^1 \cap \dot
H^3)$. The scattering part mentioned refers to the scattering for
$\psi^\pm(t)$, which follows from the Cauchy theory for the syste
\eqref{psieq}, see Theorem \ref{CT}.

Next we continue with part ii) of the Therorem \ref{MT}. From \eqref{dpsi}
we obtain the Lipschitz continuity of $\psi^\pm$ with respect $u$ as a map
from $\dot H^1$ to $L^2$. From the Cauchy theory for the system \eqref{psieq} which $\psi^\pm(t)$
obey, see Theorem \ref{CT}, we obtain the Lipschitz continuity of the $\psi^\pm(t)$ with respect to
its initial data. Finally by invoking \eqref{udif}, for each $t$, we obtain the Lipschitz continuity $u(t)$
with respect to $u_0$ in $\dot H^1$. 

\section{Appendix}

\begin{proof}[Proof of Proposition \ref{greg}]
  We write the arguments below in a qualitative fashion in order to
  have a concise argument.  However one easily sees that the argument
  below provides quantitive bounds which lead to \eqref{rtr}.

  We first read the information $u \in \dot H^2$.  Since $\Delta u \in
  L^2$ it follows that $\partial_r^2 u , \frac1r (\partial_r +
  \frac1r \partial^2_\theta) u \in L^2$, from which, using the
  equivariance property, we obtain
  \begin{equation} \label{uH2}
    \partial_r^2 u, \frac1r (\partial_r - \frac1{r}) (u_1,u_2), \frac1r \partial_r u_3 \in L^2.
  \end{equation}
  Since $|u|=1$ it follows that
  \[
  \frac{u_1 \partial_r u_1 + u_2 \partial_r u_2}r =
  -\frac{u_3 \partial_r u_3 }r \in L^2
  \]
  and by invoking $ \frac1r (\partial_r - \frac1{r}) (u_1,u_2) \in
  L^2$, we obtain $\frac{u_1^2+u_2^2}{r^2} \in L^2$.

  Next we read the information that $u \in \dot H^3$ then $\partial_x
  \Delta u, \partial_y \Delta u \in L^2$ from which, using that
  $r \partial_r = x\partial_x + y \partial_y$ it follows
  \begin{equation} \label{uH3}
    \partial_r (\partial_r^2 + \frac{1}r \partial_r-\frac{1}{r^2}) u_1, 
    \partial_r (\partial_r^2 + \frac{1}r \partial_r-\frac{1}{r^2}) u_2, 
    \partial_r (\partial_r^2 + \frac{1}r \partial_r) u_3 \in L^2
  \end{equation}
  and, using that $\partial_\theta=x\partial_y - y \partial_x$ and the
  equivariance of $\Delta u$,
  \begin{equation} \label{uH33} \frac1r (\partial_r^2 +
    \frac{1}r \partial_r-\frac{1}{r^2}) u_1, \frac1r (\partial_r^2 +
    \frac{1}r \partial_r-\frac{1}{r^2}) u_2 \in L^2.
  \end{equation}

  Since $D_r (v+iw)=0$ it follows that
  \[
  \partial_r \psi^\pm = \partial_r \left( W^\pm \cdot (v + i w)
  \right) = (\partial_r W^\pm) \cdot (v + i w)
  \]
  A direct computation gives
  \[
  \begin{split}
    \partial_r W^\pm & = \partial_r^2 u \pm \partial_r (\frac{u \times Ru}r) \\
    & = \partial_r^2 u \mp \frac{\partial_r u_3 \cdot u + u_3 \cdot \partial_r u}{r} \mp \frac{\overrightarrow{k}- u_3 \cdot u}{r^2} \\
    & = (\partial_r^2 \pm \frac{1}r \partial_r \mp \frac1{r^2}) u \mp
    \frac{1+u_3}r \partial_r u \mp \frac{\overrightarrow{k}}{r^2} +
    f^\pm u
  \end{split}
  \]
  where
  \[
  f^\pm= \mp \frac{\partial_r u_3}r \pm \frac{1+u_3}{r^2}
  \]
  We then continue with
  \[
  \begin{split}
    \partial_r \psi^\pm & = \left( (\partial_r^2 \pm
      \frac{1}r \partial_r \mp \frac{1}{r^2}) u_1, (\partial_r^2 \pm
      \frac{1}r \partial_r \mp \frac{1}{r^2}) u_2,
      (\partial_r^2 \pm \frac{1}r \partial_r) u_3 \right) \cdot (v+iw) \\
    & \mp \frac{1+u_3}r \psi_1
    \mp i \frac{(1+u_3)\psi_2}{r^2} \\
    & = F^\pm - \frac{1+u_3}{r} \psi^+
  \end{split}
  \]
  where $F^\pm \in L^2$ from \eqref{uH2}. Since $\frac{1+u_3}{r}=
  \frac1{1-u_3} \frac{u_1^2 + u_2^2}{r} \in L^1(dr)$, by using the
  integrating factor it follows that $\partial_r (e^{\int
    \frac{1+u_3}r dr} \psi^+) \in L^2$. By Sobolev embedding, in two
  dimensions, we obtain $\psi^+ \in L^4$ and since, by above,
  $\frac{1+u_3}r \in L^4$ we obtain $\partial_r \psi^+ \in L^2$.  It
  also follows that $\partial_r \psi^- \in L^2$.

  To show that $\frac{\psi^-}r \in L^2$ we write
  \[
  \frac1r W^- = \frac1r \left( (\partial_r -\frac1r) u_1, (\partial_r
    -\frac1r) u_2, \partial_r u_3 \right) + (\frac{(1+u_3)u_1}{r^2},
  \frac{(1+u_3)u_2}{r^2},\frac{u_1^2+u_2^2}{r^2})
  \]
  and use \eqref{uH2} and the fact that $\frac{u_1^2+u_2^2}{r^2} \in
  L^2$.

  Hence we have just established that $R_\pm \psi^\pm \in H^1$.  The
  procedure can be easily reversed, i.e. if $R_\pm \psi^\pm \in H^1$
  then $u \in \dot H^2$. Indeed, the additional regularity gives, by
  Sobolev embedding, that $\psi^\pm \in L^4$, hence $\frac{\psi_2}r
  \in L^4$ which implies $\frac{1+u_3}r \in L^4$. Therefore
  $\frac{1+u_3}r \psi^+ \in L^2$ and this implies that $F^\pm \in
  L^2$. This takes care of the covariant part of $\Delta u$. The part
  of $\Delta$ in the normal space is $f^\pm= \mp \frac{\Im(\psi_1 \bar
    \psi_2)}r \pm \frac{1}{1-u_3}\frac{u_1^2+u_2^2}{r^2}$ which
  belongs to $L^2$ since $\psi^\pm \in L^4$ by using Sobolev
  embeddings.

  Next we transfer third derivatives of $u$ to second derivatives for
  $\phi^\pm$ and vice-versa.  Using the above computation for
  $\partial_r \psi^+$, we have
  \[
  \begin{split}
    \partial_r^2 \psi^+ = \partial_r F^+ - \frac{1+u_3}r \partial_r
    \psi^+ - \frac{\partial_r A_2}{r} \psi^+ + \frac{1+A_2}{r^2}
    \psi^+
  \end{split}
  \]
  From \eqref{uH3} and the fact that $D_r v=D_r w=0$ we obtain that
  $\partial_r (F_1^+, F_2^+, F_3^+) \in L^2$.  Using the Sobolev
  embedding, we have that $\psi^\pm \in L^6$, hence $\frac{\psi_2}r
  \in L^6$ and $\frac{1+A_2}{r^2}=\frac1{1-A_2}\frac{|\psi_2|^2}{r^2}
  \in L^3$. Therefore, by gauging the middle terms we obtain
  \[
  \partial_r ( e^{\int \frac{1+u_3}s} \partial_r \psi^+) \in L^2
  \]
  which further gives $e^{\int \frac{1+u_3}s} \partial_r \psi^+ \in
  L^4$ (here it is important that we already have the information that
  $\partial_r \psi^+ \in L^2$). Since $\frac{1+A_2}r=\frac1{1-A_2}
  \frac{|\psi_2|^2}r \in L^4$ we obtain that in fact
  $\frac{1+u_3}r \partial_r \psi^+ \in L^2$, therefore $\partial_r^2
  \psi^+ \in L^2$.

  Next we have
  \[
  \frac{1}r \partial_r \psi^+ = \frac{1}r F^+ - \frac{1+A_2}{r^2}
  \psi^+
  \]
  From the above arguments we have that $\frac{1+A_2}{r^2} \in L^3,
  \psi^+ \in L^6$, hence the last term belongs to $L^2$. From
  \eqref{uH33} we obtain the first two components of $\frac1r F^+$ in
  $L^2$.  As for $\frac1r F_3^+$ we write
  \[
  \frac1r F_3^+= i \frac{\psi_2}r (\partial_r^2 +
  \frac{1}r \partial_r) u_3
  \]
  and, bu Sobolev embeddings, $\Delta u_3 \in L^4$ and $\frac{\psi_2}r
  \in L^4$, hence $\frac1r F_3^+ \in L^2$.  This finishes the argument
  for $\partial_r^2 \psi^+, \frac1r \partial_r \psi^+ \in L^2$ which
  implies that $R_+ \psi_+ \in H^2$.

  Next, we have
  \[
  \begin{split}
    H^- \psi^- & = (\partial_r+\frac1r) \partial_r \psi^- -\frac{4}{r^2} \psi^- \\
    & = (\partial_r+\frac1r) F^- + \frac{1+u_3}r  \partial_r \psi^+ + \frac{\partial_r A_2}{r} \psi^+  -\frac{4}{r^2} \psi^- \\
    & = \left( (\partial_r^3 - \frac{3}{r^2} \partial_r + \frac3{r^3})
      u_1, (\partial_r^3 - \frac{3}{r^2} \partial_r + \frac3{r^3})
      u_2,
      (\partial_r^3 - \frac1{r^2} \partial_r  )u_3 \right) \cdot (v+iw)\\
    & + \frac{1+u_3}r \partial_r \psi^+ + \frac{\partial_r A_2}{r}
    \psi^+ - 4 \frac{1+u_3}{r^3}(u_1,u_2,1-u_3) \cdot (v+iw)
  \end{split}
  \]
  Since
  \[
  \partial_r^3 - \frac{3}{r^2} \partial_r + \frac3{r^3} =
  \partial_r (\partial_r^2 + \frac{1}r \partial_r-\frac{1}{r^2}) -
  \frac1r (\partial_r^2 + \frac{1}r \partial_r-\frac{1}{r^2})
  \]
  we obtain from \eqref{uH3} and \eqref{uH33} that $ (\partial_r^3 -
  \frac{3}{r^2} \partial_r + \frac3{r^3}) u_1, (\partial_r^3 -
  \frac{3}{r^2} \partial_r + \frac3{r^3}) u_2 \in L^2$.  Concerning
  the term $(\partial_r^3 - \frac1{r^2} \partial_r )u_3 (v_3+iw_3)$ we
  use \eqref{uH3} so that it is enough to show that $\frac1r
  \psi_2 \partial_r^2 u_3 \in L^2$. By Sobolev embedding $\Delta u_3
  \in L^4$ and $\frac{\psi_2}r \in L^4$, hence it is enough to show
  that $\frac{1}{r^2} \psi_2 \partial_r u_3 \in L^2$ which follows
  since, by \eqref{uH2}, $\frac{\partial_r u_3}r \in L^2$ and
  $\frac{\psi_2}{r} \in L^\infty$; the last follows since $\psi^+ \in
  H^2 \subset L^\infty$ (here we refer to the two dimensional
  extension) and $\psi^- \in \dHe \subset L^\infty$.

  Next, from the argument for $\Delta \psi^+$ we have that
  $\frac{1+u_3}r \partial_r \psi^+ + \frac{\partial_r A_2}{r} \psi^+
  \in L^2$.  For the last term we use the inequality
  \[
  |\frac{1+u_3}{r^3}(u_1,u_2,1-u_3) \cdot (v+iw) | \les
  \frac{|\psi_2|^3}{r^3}
  \]
  which can then be shown to be in $L^2$ by the above arguments since
  $\frac{\psi_2}r \in L^4 \cap L^\infty$.  This finishes the proof of
  the fact that $H^- \psi^- \in L^2$, and hence of the $H^2$
  regularity of $R_- \psi^-$.

  Finally, one needs to show that if $R_\pm \psi^\pm \in H^2$ then $u
  \in \dot H^3$. This is done by backtracking the previous argument
  with the final goal of establishing \eqref{uH3} and \eqref{uH33}.
  The direct backtracking argument provides the covariant version of
  \eqref{uH3} and \eqref{uH33}, i.e. their part in the tangent
  bundle. For the part in the normal bundle, i.e. in the direction of
  $u$, one need to involve $f^\pm$ as we did before. This part
  involves one derivative less, hence it is easier. The details are
  left to the reader.
\end{proof}

\begin{lema} Assume that $f \in L^2$ . Then the following holds true
  \begin{equation} \label{pd2} \lim_{t \rightarrow \infty} \sup_{r \in
      (0,\infty)} \left | r \int_r^\infty \frac{e^{itH^-} f}{s}
      ds\right | = 0
  \end{equation}
  If $g^n=g_{\l_n,t_n}, n \in \N$ is asymptotically orthogonal to the
  constant sequence $g_{1,0}$ and $K \subset (0,+\infty)$ is compact,
  then the following holds true
  \begin{equation} \label{ps2} \lim_{n \rightarrow \infty} \| g^n f
    \|_{L^2(K)} = 0.
  \end{equation}
\end{lema}
\begin{proof}
  By the Cauchy-Schwarz inequality we have the uniform bound
  \[
  \left | r \int_r^\infty \frac{e^{itH^-} f}{s} ds\right | \lesssim
  \|e^{itH^-} f\|_{L^2} = \| f\|_{L^2}
  \]
  Hence for both \eqref{pd2} and \eqref{ps2} it suffices to prove that
  they hold for $f$ in a dense subset of $L^2$. To choose a suitable
  dense subset we first restrict ourselves to $f$ with compact
  support. Secondly, we recall that the functions $e^{2i\theta}
  e^{itH^-} f$ are solutions to the linear Schr\"odinger equations
  with initial data $e^{2i\theta} f$.  We localize $e^{2i\theta} f(r)$
  in frequency away from frequency $0$ and infinity.  If we use
  spherically symmetric multipliers this keeps the frequency localized
  functions of the form $e^{2i\theta} f(r)$ as desired. The outcome of
  this is to reduce the problem to the case when $e^{2i\theta} f(r)$
  is a Schwartz function with all moments equal to zero. This also
  implies that $f$ vanishes of infinite order at zero. Furthermore,
  due to the frequency localization away from zero we also know that
  $e^{2i\theta} f(r)$ can be represented as
  \[
  e^{2i\theta} f(r) = \Delta (e^{2i\theta} g(r))
  \]
  with $g$ in the same class.

  For initial data as described above, the solution to the linear
  Schr\"odinger equation is easily seen to satisfy bounds of the form
  \[
  | e^{it \Delta} (e^{2i\theta} f(r) )| \lesssim \frac{\la
    \frac{|t|-r}{1+\min(r,|t|)} \ra^{-N}}{\la t \ra}
  \]
  along with all its derivatives. Since $e^{it \Delta} (e^{2i\theta}
  f(r) )= e^{2i\theta} e^{itH^-} f $ this implies the bound
  \begin{equation}\label{fflow}
    | e^{itH^-} f|  \lesssim \frac{\la \frac{|t|-r}{1+\min(r,|t|)} \ra^{-N}}{\la t \ra}
  \end{equation}
  which holds for $e^{itH^-} f$ and all its $r$ derivatives.  We also
  have the representation
  \[
  e^{itH^-} f = H^- e^{itH^-} g
  \]
  where $e^{itH^-} g$ satisfies similar bounds.

  For such $f$ we now proceed to establish \eqref{pd2} and
  \eqref{ps2}.  Indeed, \eqref{ps2} follows directly from
  \eqref{fflow}. For \eqref{pd2} we use the above representation in
  terms of $g$. Integrating by parts we have
  \[
  \begin{split}
    r \int_r^\infty \frac{e^{itH^-} f}{s} ds = &\ r \int_r^\infty H^-
    (e^{itH^-} g) \cdot \frac1{s^2} s ds
    \\
    = &\ - 2 r^{-1} e^{itH^-} g(r) - \partial_r ( e^{itH^-} g(r)) + r
    \int_r^\infty e^{itH^-} g \ H^{-} \left(\frac{1}{s^2}\right) s ds
    \\
    = &\ - 2 r^{-1} e^{itH^-} g(r) - \partial_r ( e^{itH^-} g(r))
  \end{split}
  \]
  The last expression is easily seen to decay uniformly like $t^{-1}$
  in view of \eqref{fflow} applied to $g$.

\end{proof}


\begin{thebibliography}{99}

\bibitem{BG} H. Bahouri, P. Gerard, \emph{High frequency approximation
    of solutions to critical nonlinear wave equations}.
  Amer. J. Math. 121 (1999), no. 1, 131-175.

\bibitem{bv} P. Begout, A. Vargas, \emph{Mass concentration phenomena
    for the $L^2$-critical nonlinear Schr\"odinger equation}
  Trans. AMS 359 (2007), no. 11, 5257-5282.

\bibitem{Be} I. Bejenaru, \emph{On Schr\"odinger maps}, Amer. J. Math,
  no 4, 1033--1065.

\bibitem{Be2} I. Bejenaru, \emph{Global results for Schr\"odinger maps
    in dimensions $n\geq 3$}, Comm. Partial Differential Equations
  {\bf{33}} (2008), 451--477.

\bibitem{bik} I. Bejenaru, A. Ionescu, C. Kenig, \emph{Global
    existence and uniqueness of Schr\"{o}dinger maps in dimensions
    $d\geq 4$}, Adv. Math.  215 (2007), no. 1, 263--291.

\bibitem{BIKT} I. Bejenaru, A. Ionescu, C. Kenig, D. Tataru,
  \emph{Global Schr\"odinger maps}, to appear, Annals of Math.

\bibitem{BT-SSM} I. Bejenaru, D. Tataru, \emph{Near soliton evolution
    for equivariant Schr\"odinger Maps in two spatial dimensions},
  preprint

\bibitem{csu} N. Chang, J. Shatah, K. Uhlenbeck, \emph{Schr\"odinger
    maps} Comm.  Pure Appl. Math.  53 (2000), no. 5, 590--602.

\bibitem{GM} M. Grillakis, V. Stefanopoulos, \emph{Lagrangian
    formulation, energy estimates, and the Schr\"odinger map problem.}
  Comm. Partial Differential Equations 27 (2002), no. 9-10,
  1845–1877.

\bibitem{gkt1} S. Gustafson, K. Kang, T. Tsai, \emph{Schr\"odinger
    flow near harmonic maps}, Comm. Pure Appl. Math.  60 (2007),
  no. 4, 463--499.

\bibitem{gkt} S. Gustafson, K. Kang, T. Tsai, \emph{Asymptotic
    stability of harmonic maps under the Schr\"odinger flow.}, Duke
  Math. J. 145 no. 3 (2008) 537-583.

\bibitem{gnt} S. Gustafson, K. Nakanishi, T. Tsai, \emph{Asymptotic
    stability, concentration, and oscillation in harmonic map
    heat-flow, Landau-Lifshitz, and Schr\"odinger maps on $\R^2$.},
  preprint available on arxiv.

\bibitem{gk} S. Gustafson. E. Koo, \emph{Global well-posedness for 2D
    radial Schr\"odinger maps into the sphere.}

\bibitem{IoKe2} A. D. Ionescu and C. E. Kenig, \emph{Low-regularity
    Schr\"odinger maps}, Differential Integral Equations {\bf{19}}
  (2006), 1271--1300.
\bibitem{IoKe3} A. D. Ionescu and C. E. Kenig, \emph{Low-regularity
    Schr\"odinger maps, II: global well-posedness in dimensions $d\geq
    3$}, Comm. Math. Phys. {\bf{271}} (2007), 523--559.
  
\bibitem{KeMe1} C. Kenig and F. Merle, \emph{Global well-posedness,
    scattering and blow-up for the energy-critical, focusing,
    non-linear Schr\"odinger equation in the radial case.}
  Invent. Math. 166 (2006), no. 3, 645–675.

\bibitem{Ker} S. Kerani, \emph{On the defect of compactness for the
    Strichartz estimates of the Schr\"odinger equations},
  J. Differential Equations 175 (2001), no. 2, 353-392.

\bibitem{ktv} R. Killip, T. Tao, M. Visan, \emph{The cubic nonlinear
    Schr\"odinger equation in two dimensions with radial data.},
  J. Eur. Math. Soc. (JEMS) 11 (2009), no. 6, 1203–1258.


\bibitem{KS} J. Krieger, W. Schlag, \emph{Concentration compactness
    for critical wave maps}, to appear in the series "Monographs" of
  the EMS Publishing House.

\bibitem{Ga} H. McGahagan, \emph{An approximation scheme for
    Schr\"odinger maps}, Comm. Partial Differential Equations
  {\bf{32}} (2007), 375--400.
\bibitem{mv} F. Merle, L. Vega, \emph{Compactness at blow-up time for
    $L^2$ solutions of the critical nonlinear Schr\"odinger equation
    in 2D.} Internat. Math. Res. Notices 1998, no. 8, 399-425.

\bibitem{MRR} F. Merle, P. Rapha\"el, I. Rodnianski, \emph{Blow up
    dynamics for smooth equivariant solutions to the energy critical
    Schr\"odinger map}, preprint.

\bibitem{Sm1} P. Smith, \emph{Conditional global regularity of
    Schr\"odinger maps: sub-threshold dispersed energy}, preprint.
    
\bibitem{Sm2} P. Smith, \emph{Global regularity of critical Schršdinger maps: subthreshold dispersed energy},
preprint.    
    
    
\bibitem{SuSuBa} P. L. Sulem, C. Sulem, and C. Bardos, \emph{On the
    continuous limit for a system of classical spins},
  Comm. Math. Phys., {\bf{107}} (1986), 431--454.
   
    
\bibitem{tvz} T. Tao, M. Visan, X. Zhang, \emph{Minimal-mass blowup
    solutions of the mass-critical NLS.} Forum Math. 20 (2008), no. 5,
  881–919.


\end{thebibliography}
\end{document}